\pdfoutput=1
\documentclass{scrartcl}

\usepackage[section]{algorithm}
\usepackage[noend]{algpseudocode}
\newcommand{\algsize}{\footnotesize}
\newcommand{\func}[1]{\ensuremath{\textsc{{#1}}}}
\newcommand{\id}[1]{\ensuremath{\mathsf{{#1}}}}
\newcommand{\const}[1]{\textnormal{\texttt{#1}}}
\algnewcommand\Input{\item[\textbf{Input:}]}%
\algnewcommand\Output{\item[\textbf{Output:}]}%

\usepackage{topzeta}

\ifpdf
  \hypersetup{pdftitle={Computing topological zeta functions of groups, algebras, and modules, II},
    pdfauthor={Tobias Rossmann}
  }
\fi

\title{Computing topological zeta functions of groups, algebras, and modules, II}
\author{Tobias Rossmann}
\affil{\small Fakult\"at f\"ur Mathematik, Universit\"at Bielefeld, D-33501
  Bielefeld, Germany}
\date{September 2014}
\begin{document}

\maketitle
\thispagestyle{empty}

\begin{abstract}
  \small
  Building on our previous work \cite{topzeta},
  we develop the first practical algorithm for computing topological zeta functions of
  nilpotent groups, non-associative algebras, and modules.
  While we previously depended upon non-degeneracy assumptions,
  the theory developed here allows us to overcome these restrictions
  in various interesting cases, far extending the scope of \cite{topzeta}.
\end{abstract}

\blankfootnote{\indent{\itshape 2000 Mathematics Subject Classification.}
  11M41, 20F69, 14M25.

  This work is supported by the DFG Priority Programme
  ``Algorithmic and Experimental Methods in Algebra, Geometry and Number
  Theory'' (SPP 1489).}

\tableofcontents

\section{Introduction}
\label{s:intro}

\paragraph{Topological zeta functions.}
A recent addition to the theory of zeta functions of algebraic structures,
topological zeta functions of groups and algebras were introduced by du~Sautoy
and Loeser \cite{dSL04} as asymptotic invariants related to the enumeration of subobjects.
They are limits as the ``prime tends to one''  of the local subobject zeta
functions due to Grunewald, Segal, and Smith \cite{GSS88} in the same way that
topological zeta functions of polynomials due to Denef and Loeser
\cite{DL92} are limits of Igusa's local zeta functions \cite{Igu00}. 

For an informal explanation of this limit,
recall that for a finitely generated torsion-free nilpotent group $G$,
the local subgroup zeta function $\zeta_{G,p}(s)$ of $G$ at
the prime $p$ is given by the Dirichlet series $\zeta_{G,p}(s) = \sum_{e=0}^{\infty}
a_{p^e}(G) p^{-es}$, where $a_{p^e}(G)$ denotes the number of subgroups of index
$p^e$ in $G$.
Informally, the topological subgroup zeta function $\zeta_{G,\topo}(s)$ of $G$
is the rational function in $s$ obtained as the constant term of
$(1-p^{-1})^d \zeta_{G,p}(s)$ as a series in $p-1$, where $d$ is the
Hirsch length of $G$.
For example, it is well-known that
$\zeta_{\ZZ^d,p}(s) = \frac 1 {(1-p^{-s}) (1-p^{1-s}) \dotsb (1-p^{d-1-s})}
$
and we find that $\zeta_{\ZZ^d,\topo}(s) = \frac 1 {s(s-1)\dotsb(s-(d-1))}$.

A decade after their introduction, apart from a short list of examples in
\cite[\S 9]{dSL04}, topological zeta functions of groups, algebras, and modules
remained uncharted territory.
It is the purpose of the project begun in \cite{topzeta} and continued here to
change that. 

\paragraph{Central objects: toric data.}
At the heart of the present article lies
the notion of a \itemph{toric datum}. 
A toric datum consists of a half-open cone within some Euclidean space and a
finite collection of Laurent polynomials. 
We will begin our study of toric data in \S\ref{s:tdata}, where we will
also relate them to the ``cone integral data'' of du~Sautoy and Grunewald \cite{dSG00}. 
As we will see, toric data give rise to associated $p$-adic integrals (closely
related to the ``cone integrals'' from \cite{dSG00}) and to topological zeta
functions by means of a limit ``$p \to 1$''. 
Most importantly, topological zeta functions arising from the enumeration of
subgroups, subalgebras, and submodules can be expressed in terms of toric data.

In general, the computation of such zeta functions relies on the same
impractical ingredient as the computation of cone integrals: resolution of
singularities.
In suitably non-degenerate settings, explicit resolutions
can be obtained via so-called ``toric modifications'' \cite{Oka97}.
Based on such classical results from toric geometry and previous
applications to Igusa-type zeta functions (\cite{DH01}, in particular),
\cite{topzeta} provides us with explicit convex-geometric formulae for
topological zeta functions associated with toric data under non-degeneracy
assumptions.
Unfortunately, when it comes to the computation of topological subgroup, 
subalgebra, or submodule zeta functions, the practical scope of \cite{topzeta}
on its own is limited: more often than not, the non-degeneracy assumptions are
violated for examples of interest.

\paragraph{Main result.}
Our main result, Algorithm~\ref{alg:main}, is a practical algorithm 
which seeks to compute topological zeta functions associated with toric data in
favourable situations, significantly extending the applicable range of
the ideas in \cite{topzeta}.
The practicality of Algorithm~\ref{alg:main} is demonstrated by a
computer implementation \cite{zeta} which also includes a database containing a
substantial number of topological subalgebra and ideal zeta functions computed
using Algorithm~\ref{alg:main}.
These computations provide strong evidence for the intriguing features of such
topological zeta functions predicted by the conjectures in \cite[\S 8]{topzeta}.

Algorithm~\ref{alg:main} is based on a series of algebraic and convex-geometric operations,
a subset of which constitutes an algorithmic version of \cite[Thm~6.7]{topzeta}.
By adding further steps, we extend the scope of Algorithm~\ref{alg:main} beyond the
non-degeneracy assumptions of \cite{topzeta}. 
Some of these steps (such as \itemph{balancing} in \S\ref{ss:balance}) can be interpreted
within the geometric framework of toric modifications. 
Others (e.g.\ \itemph{simplification} in \S\ref{ss:simplify}) are most naturally
regarded as systematic and generalised versions of ``tricks'' for the evaluation
of $p$-adic integrals previously employed by Woodward \cite{Woo05} in a
semi-automatic fashion.
Yet others (such as \itemph{reduction} in \S\ref{ss:reduce}) are inspired by the theory of
Gr\"obner bases.

\paragraph{Previous computations.}
While Algorithm~\ref{alg:main} is the first of its kind specifically designed to
compute the \itemph{topological} zeta functions considered here, a significant number of
\itemph{local} zeta functions of groups and algebras have been previously
computed, see e.g.\ \cite{dSW08,zfarchive}.   
Although the informal definition of topological zeta functions from above
suggests that they can be deduced from local formulae, a rigorous
approach requires additional information, cf.\ \cite[\S\S 7,9.3]{dSL04}
and see \cite[Rem.\ 5.20]{topzeta}.

A significant proportion of the known local zeta functions of groups and
algebras were found by Woodward using a combination of machine-computations and
human insight \cite{Woo05}.
The number of zeta functions he managed to compute is particularly
impressive in view of the relatively elementary nature of his method which is
based on skillful (but ultimately ad hoc) applications of certain ``tricks''
\cite[\S 2.3.2]{Woo05} for computing with $p$-adic integrals.
Unfortunately, due to the reliance of his computations on human guidance, they
are difficult to reproduce.

Although, as we mentioned before, some of Woodward's ``tricks'' can be regarded
as special cases of the simplification step in \S\ref{ss:simplify}, 
the $p$-adic part of our method does not generalise his approach. 
In particular, using Algorithm~\ref{alg:main}, we managed to determine
topological zeta functions whose local versions Woodward could not compute (see
\S\ref{ss:Fil4}) and, conversely, there are examples of local zeta functions
computed by him whose topological counterparts cannot be determined using
Algorithm~\ref{alg:main}.

The techniques developed in the present article are specifically designed for
the computation of topological zeta functions of groups, algebras, and modules.
However, through our use of \cite{topzeta}, we draw heavily upon 
formulae for Igusa-type zeta functions and associated topological zeta functions
that have been obtained under non-degeneracy assumptions, see, in particular,
\cite{DL92, DH01,VZG08}.

\paragraph{Outline.}
After a brief reminder on local zeta functions of groups, algebras, and modules 
in \S\ref{s:background_gam},
we introduce toric data and associated $p$-adic integrals in \S\ref{s:tdata}.
The central section of the present article is \S\ref{s:main} which is devoted to
describing our main algorithm and its various components; details on the latter
will be provided in subsequent sections.
In \S\ref{s:balanced_and_regular}, we consider toric data which 
are balanced and regular---these two notions provide the
main link between the present article and its predecessor \cite{topzeta}.
In particular, as we will explain in \S\ref{s:topreg}, topological zeta
functions associated with regular toric data can be computed using
\cite{topzeta}.
In order to keep the present article reasonably self-contained, key facts from
\cite{topzeta} will be briefly recalled as needed.
The two remaining ingredients of Algorithm~\ref{alg:main}, namely simplification and
reduction are discussed in \S\ref{s:simplify}.
Practical aspects and the author's implementation \textsf{Zeta} \cite{zeta} of 
Algorithm~\ref{alg:main} are briefly discussed in \S\ref{s:imp}.
Finally, in \S\ref{s:app}, we consider specific examples which
illustrate key steps of Algorithm~\ref{alg:main} and which also demonstrate its
practical strength.

\subsection*{Acknowledgements}

The author would like to thank Christopher Voll for numerous inspiring
discussions.

\subsection*{{\normalfont \it Notation}}
\label{ss:notation}

The symbol ``$\subset$'' signifies not necessarily proper inclusion.
We let $\NN$, $\ZZ$, $\RR$, and $\CC$ denote the natural numbers (without
zero), integers, real and complex numbers, respectively.
We write $\NN_0 = \NN\cup\{ 0\}$.
By a $p$-adic field, we mean a finite extension of the field $\QQ_p$ of $p$-adic numbers.
Throughout this article, $k$ is a number field with ring of integers~$\fo$.
We let $K$ denote a $p$-adic field endowed with an embedding $K\supset k$.
We write $\fO_K$ and $\fP_K$ for the valuation ring of $K$ and its maximal
ideal, respectively.
We further let $\pi_K$ denote a uniformiser and let $\nu_K$ be the valuation on
$K$ with $\nu_K(\pi) = 1$; we write $\nu_K(\xx) = \bigl(\nu_K(x_1),\dotsc,\nu_K(x_n)\bigr)$
for $\xx = (x_1,\dotsc,x_n)\in K^n$.
Further write $\abs{x}_K = q_K^{-\nu_K(x)}$ and $\norm{M}_K =
\sup\bigl(\abs{x}_K : x\in M\bigr)$, where $q_K = \card{\fO_K/\fP_K}$ and
$M\subset K$.
Finally, $\mu_K$ denotes the Haar measure on $K^n$ with $\mu_K(\fO_K^n) = 1$, where
$n$ will be clear from the context.
By a non-associative algebra, we mean a not necessarily associative one.
Write $\Torus^n = \Spec\bigl(\ZZ[X_1^{\pm 1},\dotsc,X_n^{\pm 1}]\bigr)$.
For a commutative ring $R$, we identify $\Torus^n(R) = (R^\times)^n$
and write $\Torus^n_R = \Torus^n \times \Spec(R)$. 
We often write $\XX_{\phantom 1\!} \!= (X_1,\dotsc,X_n)$ and $\XX^\alpha_{\phantom
  1}= X_1^{\alpha_1}\dotsb X_n^{\alpha_n}$.

\section{Background: zeta functions of groups, algebras, and modules}
\label{s:background_gam}

The following is an abridged version of \cite[\S 2]{topzeta}. 
We investigate the following mild generalisations 
of subring, ideal, and submodule zeta functions, cf.\ \cite{Sol77,GSS88}.

\begin{defn}
  \label{d:subzeta}
  Let $R$ be the ring of integers in a number field or in a $p$-adic field.
  \begin{enumerate}
    \item
      Let $M$ be a free $R$-module of finite rank and let $\cE$ be a
      subalgebra of $\End_R(M)$.
      The \emph{submodule zeta function} of $\cE$ acting on $M$ is
      \[
      \zeta_{\cE\acts M}(s) = \sum_{n=1}^\infty \noof{\{U :
        \text{$U$ is an $\cE$-submodule of $M$ with $\idx{M:U} = n$}\}}
      \dtimes n^{-s}.
      \]
    \item
      \label{d:subzeta2}
      Let $\cA$ be a non-associative $R$-algebra whose underlying
      $R$-module is free of finite rank.
      The \emph{subalgebra zeta function} of $\cA$ is
      \[
      \zeta_{\cA}(s) = \sum_{n=1}^\infty \noof{\{\cU :
        \text{$\cU$ is an $R$-subalgebra of $\cA$ with $\idx{\cA:\cU} = n$}\}}
      \dtimes n^{-s}.
      \]
      Let $\Omega(\cA)$ be the $\fo$-subalgebra of $\End_{\fo}(\cA)$ generated by 
      $x\mapsto ax$ and $x\mapsto xa$ with $a$ ranging over $\cA$.
      Then the \emph{ideal zeta function} of $\cA$ is
      $\zeta_{\cA}^\normal(s) = \zeta_{\Omega(\cA)\acts \cA}(s)$.
    \end{enumerate}
\end{defn}

When $\fo = \ZZ$, we refer to non-associative $\ZZ$-algebras as
non-associative rings, to subalgebras as subrings, etc.
Let $\Tr_d(R)$ be the ring of upper triangular $d\times d$-matrices
over $R$.
The following result of du Sautoy and Grunewald 
is stated in the version from \cite{topzeta}.

\begin{thm}[{\cite[\S 5]{dSG00}}]
  \label{thm:coneint}
  \quad
  \begin{enumerate}
    \item
      \label{thm:coneint1}
      Let $\cA$ be a non-associative $\fo$-algebra which is free of rank $d$ as an
      $\fo$-module.
      Then there exists a finite set $\ff \subset \fo[X_{ij}^{\pm 1} : 1\le i \le j\le d]$ 
      with the following property:

      if $K \supset k$ is a $p$-adic field, then
      \begin{equation}
        \label{eq:coneint}
        \zeta_{\cA\otimes_{\fo}{\fO_K}}(s) =
        (1-q_K^{-1})^{-d}
        \int_{\bigl\{ \xx\in \Tr_d(\fO_K) : \norm{\ff(\xx)}_K\le 1\bigr\}}
        \abs{x_{11}}_K^{s-1}\dotsb \abs{x_{dd}}_K^{s-d} \dd\mu_K(\xx),
      \end{equation}
      where we identified $\Tr_d(K) \approx K^{\binom{d+1}2}$
      and we regarded $\cA\otimes_{\fo}\fO_K$ as an $\fO_K$-algebra.
  \item
    \label{thm:coneint2}
    Let $\cE$ be an $\fo$-subalgebra of $\End_{\fo}(M)$, where $M$ is a
    free $\fo$-module of rank $d$.
    Then there are Laurent polynomials as in (\ref{thm:coneint1}) such
    that the conclusion of (\ref{thm:coneint1}) holds
    for $\zeta_{(\cE\otimes_{\fo}{\fO_K}) \acts (M\otimes_{\fo}{\fO_K})}(s)$ in place of $\zeta_{\cA\otimes_{\fo}{\fO_K}}(s)$.
  \end{enumerate}
\end{thm}

As for subgroups, the enumeration of normal subgroups of a torsion-free finitely generated
nilpotent group $G$ gives rise to local  
normal subgroup zeta functions $\zeta_{G,p}^{\normal}(s)$.

\begin{thm}[{\cite[\S 4]{GSS88}}]
  \label{thm:nilpotent}
  Let $G$ be a finitely generated torsion-free nilpotent group
  with associated Lie $\QQ$-algebra $\fL(G)$ under the Mal'cev correspondence.
  Let $\cL\subset \fL(G)$ be a $\ZZ$-subalgebra which 
  is finitely generated as a $\ZZ$-module and whose $\QQ$-span is $\fL(G)$.
  Then for almost all primes $p$, we have $\zeta_{G,p}^{\phantom o}(s) =
  \zeta_{\cL\otimes_{\ZZ}\ZZ_p}^{\phantom o}(s)$ and $\zeta_{G,p}^\normal(s) = 
  \zeta_{\cL\otimes_{\ZZ}\ZZ_p}^\normal(s)$.
\end{thm}

\paragraph{Constructing Laurent polynomials from algebras and modules.}
We now recall the explicit description of $\ff$ in Theorem~\ref{thm:coneint}
given by du Sautoy and Grunewald; our exposition is equivalent to
\cite[Rem.\ 2.7(ii)]{topzeta}.
First, choose $\fo$-bases of $\cA$ or $M$ in Theorem~\ref{thm:coneint} to identify
$\cA = \fo^d$ or $M = \fo^d$ as $\fo$-modules, respectively.
We are then either given a bilinear multiplication $\beta\colon
\fo^d\otimes_{\fo}\fo^d \to \fo^d$ turning $\fo^d$ into an 
$\fo$-algebra or a finite generating set $\cM$ of $\cE \subset \Mat_d(\fo)$.
Let $R := \fo[X_{ij} :1\le i\le j \le d]$ and let $C:=[X_{ij}]_{i\le j} \in \Tr_d(R)$
with rows $C_1,\dotsc,C_d$. We think of $C$ as parameterising a generic
$\fo$-submodule of $\fo^d$ via its row span.
We extend $\beta$ to a map $R^d\otimes_R R^d\to R^d$ in the natural way.
For Theorem~\ref{thm:coneint}(\ref{thm:coneint1}),
let $\ff$ consist of the non-zero entries of
$\det(C)^{-1} \beta(C_m,C_n) \adj(C)$ for $1\le m,n\le d$;
for part (\ref{thm:coneint2}), we instead consider the entries of
$\det(C)^{-1} (CM) \adj(C)$ as $M$ ranges over $\cM$.

\section{Toric data}
\label{s:tdata}

In this section, we introduce the basic object for all
of our algorithms: toric data.
These objects are closely related to the cone integral data introduced in
\cite{dSG00}.
In particular, they also give rise to associated $p$-adic integrals and
topological zeta functions.

\subsection{Basics}
\label{ss:basics}

By a \emph{half-open cone} in $\RR^n$ we mean a set of the form
\[
\cC_0 = \bigl\{
\omega\in \RR^n : \bil{\phi_1}\omega, \dotsc, \bil{\phi_d}\omega \ge 0,
\quad
\bil{\chi_1}\omega, \dotsc, \bil{\chi_e}\omega > 0
\bigr\},
\]
where $\bil{\blank}{\blank}$ denotes the standard inner product and
$\phi_i,\chi_j\in \RR^n$.
We say that $\cC_0$ is \emph{rational} if we may choose the $\phi_i,\chi_j$
among elements of $\ZZ^n$.

\begin{defn}
  \label{d:tdatum}
  A \emph{toric datum} in $n$ variables over $k$ is a
  pair $\cT = (\cC_0;\ff)$ consisting of a half-open rational cone
  $\cC_0\subset \Orth^n$ and a finite family $\ff = (f_1,\dotsc,f_r)$ of
  Laurent polynomials $f_1,\dotsc,f_r \in k[X_1^{\pm 1},\dotsc,X_n^{\pm 1}]$.
\end{defn}

We often write $(\cC_0;f_1,\dotsc,f_r)$ instead of $(\cC_0;\ff)$.
Furthermore, we usually omit the references to $n$ and $k$.
The non-negativity assumption $\cC_0 \subset \Orth^n$ is included to 
ensure the convergence of the integrals in \S\ref{ss:assoczeta} below.
As we will explain in Remark~\ref{r:special}, toric data provide us with a
convenient formalism for describing the domain of integration for the integrals
in Theorem~\ref{thm:coneint}. 
Recall the notational conventions from p.\!~\pageref{ss:notation}.
\begin{notation}
  Given $\cT = (\cC_0;\ff)$ as in Definition~\ref{d:tdatum}
  and a $p$-adic field $K\supset k$,
  we write
  $$\cT_K := \Bigl\{ \xx\in \Torus^n(K) : \nu_K(\xx) \in \cC_0,
  \norm{\ff(\xx)}_K \le 1 \Bigr\}.$$
\end{notation}

\paragraph{Toric data and cone conditions.}
For an explanation of our terminology, suppose that $\cC_0$ is closed.
Then each of the conditions $\nu_K(\xx)\in \cC_0$ and $\norm{\ff(\xx)}_K \le 1$
can be expressed as a conjunction of finitely many divisibility conditions
$\divides{v(\xx)}{w(\xx)}$, where $v,w\in k[\XX]$.
Indeed, let $\cC_0 = \{ \omega\in \RR^n : \bil{\phi_1}\omega,
\dotsc, \bil{\phi_d}\omega\ge 0\}$ for $\phi_1,\dotsc,\phi_d \in \ZZ^n$ 
and $\ff = (f_1,\dotsc,f_r)$.
It is easy to see that $(\cC_0;\,\ff)_K = \bigl(\Orth^n; \, f_1,\dotsc,f_r,
\XX^{\phi_1},\dotsc,\XX^{\phi_d}\bigr)_K$.
Next, for a Laurent polynomial $g = \XX^{-\gamma}  g^+$ with $\gamma \in
\NN_0^n$ and $g^+\in \fO_K[\XX]$, the condition $\abs{g(\xx)}_K\le
1$ (where $\xx\in \Torus^n(K) \cap \fO_K^n = (\Orth^n;\emptyset)_K$) is equivalent to
$\divides{\xx^{\gamma}}{g^+(\xx)}$.

Hence, for $\cC_0$ closed, a toric datum gives rise to a special
case of a ``cone condition'' as defined in \cite[Def.\ 1.2(1)]{dSG00}.
Moreover,  if $\cD_K$ denotes the set of $K$-points of such a cone condition,
then $(\cC_0;\ff)_K = \cD_K \cap \Torus^n(K)$.

\begin{defn}
  \label{d:trivial}
  A toric datum $(\cC_0;\ff)$ is \emph{trivial} if $\cC_0 = \emptyset$.
\end{defn}

\subsection{Zeta functions associated with toric data}
\label{ss:assoczeta}

Let $\beta\in\Mat_{m\times n}(\NN_0)$ with rows $\beta_1,\dotsc,\beta_m$.
Given a toric datum $\cT$ in $n$ variables over $k$ and a $p$-adic field $K\supset k$,
we consider the ``zeta function'' defined by 
\begin{equation}
  \label{eq:basic_integral}
  \Zeta^{\cT,\beta}_K(s_1,\dotsc,s_m) :=
  \int_{\cT_K}
  \abs[\big]{\xx^{\beta_1}}_K^{s_1} \dotsb\abs[\big]{\xx^{\beta_m}}_K^{s_m}
  \dd\mu_K(\xx),
\end{equation}
where $s_1,\dotsc,s_m\in \CC$ with $\Real(s_j) \ge 0$;
convergence is guaranteed by the non-negativity assumptions
$\cC_0\subset\Orth^n$, $\beta_1,\dotsc,\beta_m\in \NN_0^n$.
We note that  $\Zeta_K^{\cT,\beta}(s_1,\dotsc,s_m)$ is a
special case of the zeta functions studied in \cite[\S 4]{topzeta}.

\begin{rem}[Local subalgebra and submodule zeta functions]
  \label{r:special}
  Disregarding factors of the form $(1-q_K^{-1})^{\pm d}$,
  Theorem~\ref{thm:coneint} shows that zeta functions associated with toric data
  generalise local zeta functions arising from the enumeration of (normal)
  subgroups,  subalgebras, and submodules as in \S\ref{s:background_gam},
  cf.\ \cite[Rem.\ 4.12]{topzeta}.
  Indeed, let $\cA$ (or $M$) in Theorem~\ref{thm:coneint} have rank $d$ and let $n =
  d(d+1)/2$.
  We identify $\Tr_d \approx \AA^n$ via $(x_{ij}) \mapsto
  (x_{11},\dotsc,x_{1d},x_{22},\dotsc, x_{dd})$
  and let $\ff\subset \fo[X_1^{\pm 1},\dotsc,X_n^{\pm 1}]$ be an
  associated family of Laurent polynomials as in Theorem~\ref{thm:coneint};
  see the end of \S\ref{s:background_gam} for an explicit construction.
  Let $\cT = (\Orth^n;\ff)$ and let $\beta\in\Mat_{d\times n}(\NN_0)$ be the matrix
  whose $j$th row corresponds to the elementary matrix with entry $1$ in position
  $(j,j)$ under the above isomorphism $\Tr_d \approx \AA^n$. 
  Then $\Zeta^{\cT,\beta}_K(s_1,\dotsc,s_m)$ specialises
  to the integral in Theorem~\ref{thm:coneint} via 
  $(s_1,\dotsc,s_m) \mapsto (s-1,\dotsc,s-d)$.
\end{rem}

\paragraph{Relationship with cone integrals.}
In \S\ref{ss:basics}, we explained how toric data give rise to cone
conditions from \cite{dSG00} (at least when $\cC_0$ is closed).
In the same spirit, we may regard the integrals in \eqref{eq:basic_integral} as
special cases of (multivariate versions of) the cone integrals in \cite{dSG00};
the pair $(\cT,\beta)$ takes the place of the ``cone integral data'' in \cite{dSG00}.
One of the two important special features of \eqref{eq:basic_integral} compared
with cone integrals is that we insist on left-hand sides in the divisibility
conditions describing the domain of integration being monomial (just as they are
in Theorem~\ref{thm:coneint}).
By expressing such divisibility conditions in terms of Laurent
polynomials as in \S\ref{ss:basics}, we naturally adopt a ``toric'' point
of view.
This perspective will prove to be especially useful in combination with the
second key feature of the integrals \eqref{eq:basic_integral}, namely the
presence of a not necessarily closed ambient half-open cone $\cC_0$.
Focusing exclusively on integrals of the shape \eqref{eq:basic_integral} allows 
us to develop specialised techniques for manipulating and evaluating them.

\paragraph{Evaluation in theory: ``explicit formulae''.}
Consider the $K$-indexed family of zeta functions $\Zeta_K^{\cT,\beta}$
defined in terms of a toric datum $\cT$ and a matrix $\beta$ 
in \eqref{eq:basic_integral}. 
Going back to work of Denef~\cite{Den87} and du~Sautoy and Grunewald \cite{dS00},
using powerful but typically impractical techniques such as resolution of 
singularities, it can be shown (cf.\ \cite[Ex.~5.11(vi)]{topzeta}) that there 
are finitely many $W_i(\qq,\tee_1,\dotsc,\tee_m)\in
\QQ(\qq,\tee_1,\dotsc,\tee_m)$ and 
$k$-varieties $V_i$ for $i\in I$, say, with the following property:
if $K\supset k$ is a $p$-adic field, then,
unless $\fp_K := \fo \cap \fP_K$ belongs to some finite exceptional set
(depending on $\cT$ only), 
\begin{equation}
  \label{eq:basic_denef}
  \Zeta_K^{\cT,\beta}(s_1,\dotsc,s_m) = 
  \sum_{i\in I} \noof{\bar V_i(\fO_K/\fP_K)} \dtimes W_i(q_K^{\phantom{s_1}},q_K^{-s_1},\dotsc,q_K^{-s_m}),
\end{equation}
where $\bar V_i$ is the reduction modulo $\fp_K$ of $V_i$.
We understand the task of ``computing'' the $\Zeta_K^{\cT,\beta}(s_1,\dotsc,s_m)$ to
be the explicit construction of $V_i$ and $W_i(\qq,\tee_1,\dotsc,\tee_m)$ as in \eqref{eq:basic_denef}.

\paragraph{Topological zeta functions.}
As originally observed by Denef and Loeser \cite{DL92} for Igusa's local zeta
function, given an ``explicit formula'' as in \eqref{eq:basic_denef}, under 
additional assumptions regarding the shapes of the
$W_i(\qq,\tee_1,\dotsc,\tee_m)$ (see \S\ref{ss:M}),
we may ``pass to the limit $q_K \to 1$'' and obtain the associated topological
zeta function  
\begin{equation}
  \label{eq:basic_top}
  \Zeta^{\cT,\beta}_{\topo}(s_1,\dotsc,s_m) =
  \sum_{i\in I}
  \Euler(V_i(\CC)) \dtimes \red{W_i}(s_1,\dotsc,s_m) \in \QQ(s_1,\dotsc,s_m);
\end{equation}
here $\Euler(V_i(\CC))$ denotes the topological Euler characteristic 
with respect to any embedding of $k$ into $\CC$ and 
$\red{W_i}(s_1,\dotsc,s_m) \in \QQ(s_1,\dotsc,s_m)$ is
the constant term of $W_i(q_K^{\phantom{s_1}},q_K^{-s_1},\dotsc,q_K^{-s_m})$,
formally expanded as a series in $q_K-1$.
A particularly noteworthy consequence of \cite{DL92} is 
that the right-hand side of \eqref{eq:basic_top} is independent of the choice
of the family $\bigl(V_i,W_i(\qq,\tee_1,\dotsc,\tee_m)\bigr)_{i\in I}$ in \eqref{eq:basic_denef}.

\paragraph{Evaluation in practice: non-degeneracy.}
Let $\cT = (\cC_0;\ff)$ be a toric datum over $k$.
If $\ff$ is non-degenerate relative to $\cC_0$
in the sense of \cite[Def.\ 4.2(i)]{topzeta}, 
then \cite[Thm~4.10]{topzeta} yields an effective version of
\eqref{eq:basic_denef} in the sense that
it provides \itemph{explicit} descriptions
of varieties $V_i$ and rational functions $W_i(\qq,\tee_1,\dotsc,\tee_m)$
in terms of convex-geometric data associated with various cones and polytopes
attached to $\cT$ and $\beta$, see Theorem~\ref{thm:regular_padic}.
While further computations involving these objects might still be expensive or
even infeasible in large dimensions, if available, they are however much more 
useful than formulae 
obtained using general resolution algorithms; the latter are usually only
practical for $n \le 3$.
Recall from Remark~\ref{r:special} that for the computation of subalgebra or
submodule zeta functions, $n = d(d+1)/2$, where $d$ is the additive rank of the
object under consideration.

If $\ff$ is even globally non-degenerate (see \cite[Def.\ 4.2(ii)]{topzeta}), then the Euler
characteristics in \eqref{eq:basic_top} can be expressed in terms of mixed
volumes via the Bernstein-Kushnirenko-Khovanskii Theorem, yielding an
effective form of \eqref{eq:basic_top}, see \cite[Thm~6.7]{topzeta}.
The theory underpinning Algorithm~\ref{alg:main} in \S\ref{s:main} (to be
developed in the present article) draws upon and extends this result to overcome
certain instances of degeneracy. 

\paragraph{Topological subalgebra and submodule zeta functions.}
Under rather weak technical assumptions on the $W_i(\qq,\tee_1,\dotsc,\tee_m)$
appearing in \eqref{eq:basic_top}, passing from local to topological zeta functions
commutes with affine specialisations of the variables $s_1,\dotsc,s_m$, see \cite[Rem.\ 5.15]{topzeta}.
In particular, we obtain univariate versions of \eqref{eq:basic_denef}
and \eqref{eq:basic_top} arising from the integrals in
Theorem~\ref{thm:coneint} and thus rigorous definitions
of topological subalgebra and submodule zeta functions,
see \cite[Def.\ 5.17]{topzeta} or \S\ref{ss:M} below.

\section{The main algorithm}
\label{s:main}

In this section, we give a high-level description of an algorithm
which seeks to compute topological zeta functions associated with toric data.
The main application that we have in mind is the computation of
topological subalgebra and submodule zeta functions via the univariate
specialisations explained in Remark~\ref{r:special}.

\paragraph{The algorithm.}
We suppose that we are given a toric datum $\cT^0 = (\cC_0;\ff)$ in $n$
variables over $k$---in practice, we are primarily interested in the case where $\ff$ is a family of
Laurent polynomials arising from Theorem~\ref{thm:coneint} and $\cC_0 = \Orth^n$ (where
$n = d(d+1)/2$) as in Remark~\ref{r:special}.
Given $\cT^0$ and  $\beta\in \Mat_{m\times n}(\NN_0)$,
the function \func{TopologicalZetaFunction} (Algorithm~\ref{alg:main}) attempts
to compute the topological zeta function
$\Zeta_{\topo}^{\cT^0,\beta}(\ess_1,\dotsc,\ess_m) \in
\QQ(\ess_1,\dotsc,\ess_m)$ associated with the integrals
$\Zeta_K^{\cT^0,\beta}(s_1,\dotsc,s_m)$ in \eqref{eq:basic_integral};
see \S\ref{ss:toptcc} for a rigorous definition of $\Zeta_{\topo}^{\cT^0,\beta}(\ess_1,\dotsc,\ess_m)$.
We note that from now on, we use bold face letters $\ess_j$ to distinguish
variables over $\QQ$ from the complex numbers $s_j$ in \S\ref{ss:assoczeta}.

\begin{algorithm}
  \caption{$\func{TopologicalZetaFunction}(\cT^0, \beta)$}
  \label{alg:main}
  \begin{algorithmic}[1]
    \algsize
    \Input a toric datum $\cT^0$ in $n$ variables over $k$, a matrix
    $\beta \in \Mat_{m\times n}(\NN_0)$
    \Output 
    the topological zeta function $\Zeta_{\topo}^{\cT^0,\beta}(\ess_1,\dotsc,\ess_m) \in \QQ(\ess_1,\dotsc,\ess_m)$ or \const{fail}

    \State $\id{unprocessed} \gets [\cT^0]$, $\id{regular} \gets [\,]$,
    \Comment Stage I
    \label{alg:init_main}

    \While{$\id{unprocessed}$ is non-empty}
    \label{alg:begin_main_loop}
    \State remove an element $\cT$ from \id{unprocessed}
    \State $\cT \gets \func{Simplify}(\cT)$
    \label{alg:simplification_step}
    
    \If{$\cT$ is not balanced}
    \State $\id{new} \gets \func{Balance}(\cT)$
    \ElsIf{$\cT$ is regular}
    \State add $\cT$ to \id{regular}
    \State $\id{new} \gets [\,]$
    \Else
    \State $\id{new} \gets \func{Reduce}(\cT)$
    \label{alg:begin_reduce}
    \If{$\id{new} = \const{fail}$}
    {\Return \const{fail}}
    \label{alg:end_reduce}
    \label{alg:reduce_failed}
    \EndIf
    \EndIf
    \State add the non-trivial elements of \id{new} to \id{unprocessed}
    \EndWhile
    \label{alg:end_main_loop}

    \State \Return $\sum\limits_{\cT\in
      \id{regular}}\func{EvaluateTopologically}(\cT,\beta)$
    \label{alg:summing_up}
    \Comment Stage II
  \end{algorithmic}
\end{algorithm}

We now explain the structure of Algorithm~\ref{alg:main} and the roles
played by the functions \func{Simplify}, \func{Balance}, \func{Reduce}, and
\func{EvaluateTopologically}. 
Details will be given in the following sections. 

\paragraph{Stage I: the main loop.}
During the first stage (lines~\ref{alg:init_main}--\ref{alg:end_main_loop}) of
Algorithm~\ref{alg:main}, we maintain two lists, \id{unprocessed} and
\id{regular}, of toric data.
The essential point here is that unless the execution of
Algorithm~\ref{alg:main} is aborted in line~\ref{alg:reduce_failed},
each iteration of the loop in
lines~\ref{alg:begin_main_loop}--\ref{alg:end_main_loop} preserves the following
property: 
\begin{itemize}
\item[(\textlabel{$\clubsuit$}{LOOPINV})]
There exists a finite $S \subset \Spec(\fo)$ such that if $K\supset k$ is a
$p$-adic field with $\fo\cap\fP_K\not\in S$,
then $\cT^0_K$ is the disjoint union of all $\cT_K$ with $\cT$ ranging
over $\id{unprocessed} \sqcup \id{regular}$.
\end{itemize}
In particular, we always have
$\Zeta^{\cT^0,\beta}_{\topo}(\ess_1,\dotsc,\ess_m) = \sum_{\cT}
\Zeta^{\cT, \beta}_{\topo}(\ess_1,\dotsc,\ess_m)$, where
$\cT$ again ranges over $\id{unprocessed} \sqcup \id{regular}$ (Lemma~\ref{lem:topadd}).

The central notions, to be defined in \S\ref{s:balanced_and_regular}, 
featuring in the while-loop in Algorithm~\ref{alg:main} are those of
\itemph{balanced} and \itemph{regular} toric data.
The function \func{Balance} (see \S\ref{ss:balance}) 
takes as input an arbitrary toric datum and
returns a distinguished family of associated balanced ones.
Among balanced toric data, regular ones constitute a subclass
which is intimately related to the concept of non-degeneracy
used in \cite{topzeta}.
Namely, given a toric datum $\cT = (\cC_0;\ff)$, the family
$\ff$ is non-degenerate relative to $\cC_0$ in the sense of \cite[Def.\
4.2(i)]{topzeta} if and only if each element of $\func{Balance}(\cT)$ is regular
(Proposition~\ref{prop:nd_vs_regular}).
Our objective during the first stage of Algorithm~\ref{alg:main} is to
successively modify and decompose toric data until, hopefully, at
some point all toric data under consideration will be regular.
In the reduction step (lines~\ref{alg:begin_reduce}--\ref{alg:end_reduce}),
given a balanced toric datum $\cT$ that fails to be regular,
the function \func{Reduce} (see \S\ref{ss:reduce}) attempts to mitigate this
failure of regularity by allowing us to replace $\cT$ by a suitable family of toric data.
This step is  supported by the function \func{Simplify}
(see \S\ref{ss:simplify}) which removes certain redundancies from its input.

\paragraph{Stage II: topological evaluation.}
After successful termination of the while-loop in Algorithm~\ref{alg:main}, the
computation of the topological zeta function
$\Zeta^{\cT^0,\beta}_{\topo}(\ess_1,\dotsc,\ess_m)$  is reduced to
computing the topological zeta functions associated with a (possibly
large) number of {regular} toric data via the function
\func{EvaluateTopologically} (see \S\ref{ss:EvaluateTopologically}).
Given a regular toric datum, by the aforementioned connection between
regularity and non-degeneracy and as previously indicated in
\S\ref{ss:assoczeta}, \cite{topzeta} provides us with explicit 
convex-geometric $p$-adic formulae.
Using \cite[\S 6]{topzeta}, we may then deduce a formula for 
$\Zeta^{\cT^0,\beta}_{\topo}(\ess_1,\dotsc,\ess_m)$ in terms of the
topological Euler characteristics of certain subvarieties of tori.
However, in contrast to the assumption of global non-degeneracy in
\cite[Thm~6.7]{topzeta}, regularity alone does not suffice to provide us with
explicit formulae for these Euler characteristics.

\begin{rem}
  \quad
  \begin{enumerate}
  \item
    Write $(\cD_0;\bm g) = \func{Simplify}(\cT^0)$.
If $\bm g$  is non-degenerate relative to $\cD_0$ in the sense of \cite[Def.\ 4.2(i)]{topzeta},
then the first stage of Algorithm~\ref{alg:main} will always succeed (see Remark~\ref{r:ET}(\ref{r:ET2})).
If $\bm g$ is even globally non-degenerate (see
\cite[Def.\ 4.2(ii)]{topzeta}), then \func{TopologicalZetaFunction} constitutes
an algorithmic version of \cite[Thm\ 6.7]{topzeta} for the class of topological
zeta functions considered here.
However, Algorithm~\ref{alg:main} can do much more:
the reduction (lines~\ref{alg:begin_reduce}--\ref{alg:end_reduce}) and
simplification (line~\ref{alg:simplification_step}) steps allow
it to overcome certain instances of degeneracy.
When it comes to the practical computation of topological subalgebra and
submodule zeta functions, the computations documented in \S\ref{s:app}
demonstrate that Algorithm~\ref{alg:main} substantially extends the 
scope of \cite{topzeta}.
It is however important to note that a possible point of failure remains in
Algorithm~\ref{alg:main}, namely as part of the reduction step in line~\ref{alg:reduce_failed}.
  \item
    The separation of Algorithm~\ref{alg:main} into two stages might seem artificial.
    It is justified by the aforementioned chance of failure of
    Algorithm~\ref{alg:main}
    and the observations that the final evaluation in
    line~\ref{alg:summing_up} is often the most computationally expensive
    step.
  \item
    The first stage of Algorithm~\ref{alg:main} is $p$-adic in nature in the
    sense that it consists entirely of manipulations of $p$-adic
    integrals. In contrast, the second  stage is inherently topological.
    Finding a practically useful $p$-adic version of the second stage and thus
    a practical method for computing associated local zeta functions
    is a natural direction for future research. 
  \end{enumerate}
\end{rem}

\section{Balanced and regular toric data}
\label{s:balanced_and_regular}

\subsection{Background: cones, polytopes, and polynomials}
\label{ss:newton}

The following summary of well-known material is
based upon \cite[\S\S 3.1, 3.3, 4.1]{topzeta}.

\paragraph{Cones.}
A (closed) \emph{cone} in $\RR^n$ is a
set of the form
$\cC = \bigl\{
\omega\in \RR^n : \bil{\phi_1}\omega, \dotsc, \bil{\phi_d}\omega \ge 0
\}$, where $\phi_1,\dotsc,\phi_d \in \RR^n$.
Equivalently, cones in $\RR^n$ are sets 
of the form $$\cone(\varrho_1,\dotsc,\varrho_r) := \Orth \dtimes \varrho_1 + \dotsb + \Orth
\dtimes \varrho_r,$$ where $\varrho_1,\dotsc,\varrho_r\in \RR^n$;
we call the $\varrho_1,\dotsc,\varrho_r$ a system of \emph{generators} of
$\cone(\varrho_1,\dotsc,\varrho_r)$.
A half-open cone as defined in \S\ref{ss:basics} is precisely a set of 
the form $\cC\setminus(\cC_1\cup\dotsb\cup\cC_r)$, where $\cC$ is a cone,
$r\ge 0$, and $\cC_1,\dotsc,\cC_r$ are faces of $\cC$.
In particular, cones are half-open cones.

\paragraph{Polytopes and normal cones.}
We insist that faces of a polytope be non-empty. 
For a non-empty polytope $\cP \subset \RR^n$ and $\omega\in \RR^n$, 
let $\face_\omega(\cP)$ denote the face of $\cP$ where $\bil\blank\omega$
attains its minimum.
If $\cQ \subset \RR^n$ is another non-empty polytope, then $\face_\omega(\cP + \cQ) =
\face_\omega(\cP) + \face_\omega(\cQ)$.
The \emph{normal cone} of a face $\tau \subseteq \cP$ is
the relatively open cone
$\NormalCone_\tau(\cP) = \{ \omega \in \RR^n : \face_\omega(\cP) = \tau \}$
of dimension $n- \dim(\tau)$.
The $\NormalCone_{\tau}(\cP)$ partition $\RR^n$ into relatively open cones.
The normal cones of faces of $\cP + \cQ$ coincide with the non-empty
sets of the form
$\NormalCone_\tau(\cP) \cap \NormalCone_\upsilon(\cQ)$ for faces $\tau \subset
\cP$ and $\upsilon \subset \cQ$,
corresponding to the unique decomposition of a face of $\cP + \cQ$ into a 
sum of the form $\tau + \upsilon$. 

\paragraph{Newton polytopes and initial forms.}
Let $f = \sum_{\alpha\in \ZZ^n} c_\alpha \XX^\alpha \in k[\XX^{\pm 1}]$,
where $c_\alpha\in k$.
The \emph{support} of $f$ is $\supp(f) = \{ \alpha \in \ZZ^n : c_\alpha \not= 0\}$.
The \emph{Newton polytope} $\Newton(f)$ of $f$ is the convex hull
of $\supp(f)$ within $\RR^n$.
For $f,g\in k[\XX^{\pm 1}]$,
we have $\Newton(fg) = \Newton(f) + \Newton(g)$.
For $\omega \in \RR^n$, the \emph{initial form} $\init_\omega(f)$ of $f$ in
the direction $\omega$ is the sum of those $c_\alpha \XX^\alpha$ with
$\alpha \in \supp(f)$ where $\bil \alpha \omega$ attains its minimum.
If $f\not= 0$,
then $\face_\omega(\Newton(f)) = \Newton(\init_\omega(f))$ for $\omega\in
\RR^n$. 
The equivalence classes on $\RR^n$ 
of $\omega \sim \omega' :\iff \init_{\omega}(f) = \init_{\omega'}(f)$
are the normal cones of the faces of $\Newton(f)$.

\subsection{Balanced toric data}
\label{ss:balance}

\begin{defn}
  \label{d:balanced}
  \quad
  \begin{enumerate}
  \item
  \label{d:balanced1}
    Let $\cC_0 \subset \RR^n$ be an arbitrary subset.
    We say that $f\in k[\XX^{\pm 1}]$ is \emph{$\cC_0$-balanced}
    if $\init_\omega(f)$ is constant as $\omega$ ranges over $\cC_0$.
    If $\cC_0\not= \emptyset$, we denote this common initial form by $\init_{\cC_0}(f)$.
  \item
  \label{d:balanced2}
  A toric datum $(\cC_0;f_1,\dotsc,f_r)$ is
  \emph{balanced} if each $f_i$ is $\cC_0$-balanced.
  \end{enumerate}
\end{defn}

\begin{ex}
  \label{ex:balanced}
  Let $n = 2$ and $k$ be arbitrary.
  Define $f_1 = X_1^{-1} - X_2^{-1}$ and $f_2 = X_1^{-2} - X_2^{-2}$.
  Then $(\Orth^2;f_1,f_2)$ is not balanced. For instance, $\init_{(1,0)}(f_1) =
  X_1^{-1}$ but $\init_{(0,1)}(f_1) = -X_2^{-1}$.
  Let $\cC_0 = \{ (\omega_1,\omega_2)\in \Orth^2 : \omega_1 > \omega_2\}$.
  Then $(\cC_0;f_1,f_2)$ is balanced with $\init_{\cC_0}(f_1) = X_1^{-1}$ and
  $\init_{\cC_0}(f_2) = X_1^{-2}$.
\end{ex}

\begin{lemma}
  \label{lem:balanced_crit}
  Let $\cT = (\cC_0;f_1,\dotsc,f_r)$ be a toric datum with $f_1\dotsb
  f_r\not= 0$. Let $\cN = \Newton(f_1\dotsb f_r)$.
  Then $\cT$ is balanced if and only if
  $\cC_0\subset \NormalCone_\tau(\cN)$ for some face $\tau \subset \cN$.
\end{lemma}
\begin{proof}
  We may assume that $\cC_0 \not= \emptyset$.
  Suppose that $\cT$ is balanced.
  Then, for $1\le i\le r$, since $f_i$ is $\cC_0$-balanced, there exists a face
  $\tau_i\subset \Newton(f_i)$ with $\cC_0 \subset
  \NormalCone_{\tau_i}(\Newton(f_i))$.
  Hence, $\emptyset\not= \cC_0 \subset \bigcap_{i = 1}^r
  \NormalCone_{\tau_i}(\Newton(f_i))$ and 
  $\tau := \sum_{i=1}^r \tau_i$ is the desired face 
  of $\sum_{i=1}^r \Newton(f_i) = \cN$.
  Conversely, let $\cC_0 \subset \NormalCone_\tau(\cN)$ and
  write $\tau = \sum_{i=1}^r \tau_i$ for faces $\tau_i \subset\Newton(f_i)$.
  Then $\cC_0 \subset \NormalCone_\tau(\cN) \subset
  \NormalCone_{\tau_i}(\Newton(f_i))$ for $1\le i\le r$, whence $\cT$ is balanced.
\end{proof}

The following notion will be used to show that 
Algorithm~\ref{alg:main} preserves (\ref{LOOPINV}) from p.\!~\pageref{LOOPINV}.
\begin{defn}
  \label{d:partition}
  Let $\cT$ be a toric datum.
  A finite collection $\mathfrak C$ of toric data 
  is a \emph{partition} of $\cT$ if there exists a
  finite $S\subset\Spec(\fo)$ such that 
  if $K\supset k$ is a $p$-adic field with $\fo\cap\fP_K\not\in S$, then
  $\cT_K = \bigcup_{\cT'\in \mathfrak C} \cT'_K$ (disjoint).
\end{defn}

We may thus rephrase (\ref{LOOPINV}) by stating that
$\id{unprocessed}\sqcup\id{regular}$ is a partition of $\cT^0$.

\paragraph{The function \textrm{\func{Balance}}.}
Given a toric datum $(\cC_0;f_1,\dotsc,f_r)$, the function
\func{Balance} produces a partition consisting of balanced toric
data as follows.
Let $I = \{i: f_i\not= 0\}$ and $\cN = \Newton(\prod_{i\in
  I}f_i)$.
Then let $\func{Balance}(\cC_0;f_1,\dotsc,f_r)$ return the collection
of $\bigl(\cC_0\cap\NormalCone_\tau(\cN);f_1,\dotsc,f_r\bigr)$ for faces $\tau
\subset\cN$ with $\cC_0\cap\NormalCone_\tau(\cN)\not= \emptyset$.
Note that each toric datum
$\bigl(\cC_0\cap\NormalCone_\tau(\cN);f_1,\dotsc,f_r\bigr)$ is balanced by Lemma~\ref{lem:balanced_crit}.

\subsection{Regular toric data}
\label{ss:regular}

Let $\bar k$ be an algebraic closure of $k$.

\begin{defn}
  \label{d:regular}
  We say that a balanced toric datum $(\cC_0;f_1,\dotsc,f_r)$ 
  is \emph{regular} if either $\cC_0 =
  \emptyset$ or the following condition is satisfied:
  \begin{itemize}
  \item[]
  For all $J \subset\{1,\dotsc,r\}$,
  if $\uu\in \Torus^n(\bar k)$ satisfies $\init_{\cC_0}(f_j)(\uu) = 0$ for all $j\in J$,
  then the Jacobian matrix $\Bigl[ \frac{\partial\!\init_{\cC_0}(f_j)}{\partial
    X_i}(\uu)\Bigr]_{i=1,\dotsc,n;j\in J}$ has rank $\card J$.
\end{itemize}
We say that $(\cC_0;f_1,\dotsc,f_r)$ is \emph{singular} if it is balanced but not regular.
\end{defn}

\begin{ex}
  \label{ex:regular}
  The toric datum $(\cC_0;f_1,f_2)$ in Example~\ref{ex:balanced} is trivially
  regular since both initial forms $\init_{\cC_0}(f_1)$ and $\init_{\cC_0}(f_2)$
  are Laurent monomials and hence do not vanish on $\Torus^2(\bar k)$.
  Let $\cC_0' = \{ (\omega_1,\omega_2)\in \Orth^2 : \omega_1 = \omega_2\}$.
  Then $(\cC_0';f_1,f_2)$ is balanced but singular.
  Indeed, the initial forms are $\init_{\cC_0'}(f_1) = f_1$ and $\init_{\cC_0'}(f_2) =
  f_2$ and the condition in Definition~\ref{d:regular} is violated  on the
  subvariety of $\Torus^2_k$ defined by $X_1 = X_2$ for $J = \{1,2\}$.
\end{ex}

By definition, a balanced toric datum $(\cC_0;\ff)$ with $0\not\in \ff$
is regular if and only if $\ff$ is non-degenerate relative to $\cC_0$
in the sense of \cite[Def.\ 4.2(i)]{topzeta}.
More generally, the following holds by construction.
\begin{lemma}
  \label{prop:nd_vs_regular}
  Given 
  $(\cC_0;\ff)$ with $0\not\in\ff$, the family $\ff$ is non-degenerate relative to $\cC_0$ in the sense of
  \cite[Def.\ 4.2(i)]{topzeta} if and only if each element of
  $\func{Balance}(\cC_0;\ff)$ is regular.
  \qed
\end{lemma}

\paragraph{Testing regularity.}
As a part of Algorithm~\ref{alg:main}, we need to test regularity of toric data.
This can be carried out using Gr\"obner bases computations as follows.
Let $(\cC_0;f_1,\dotsc,f_r)$ be a balanced toric datum with
$\cC_0 \not= \emptyset$.
Write $g_i = \init_{\cC_0}(f_i)$ and $M_J = \Bigl[ \frac{\partial
  g_j}{\partial X_i}\Bigr]_{i=1,\dotsc,n;j\in J}$.
By the weak Nullstellensatz,
$(\cC_0;\ff)$ is regular if and only if for each $J\subset \{ 1,\dotsc,r\}$,
the Laurent polynomials $g_j$ for $j\in J$ together with the $\card J\times
\card J$-minors of $M_J$ generate the unit ideal of $k[\XX^{\pm 1}]$.
For practical computations, it is convenient to rephrase the latter condition in
terms of the polynomial algebra $k[\XX]$. 
Thus, since regularity of $(\cC_0;f_1,\dotsc,f_r)$ is invariant under rescaling
of the $f_i$ by Laurent monomials (cf.\ \cite[Rem.\ 4.3(ii)]{topzeta}),
we may assume that $g_1,\dotsc,g_r \in k[\XX]$ are polynomials. 
It follows that $(\cC_0;f_1,\dotsc,f_r)$ is regular if and only if 
for all $J$, the monomial $X_1\dotsb X_n$ is contained in the radical of the
ideal generated by all $g_j$ ($j\in J$) and the $\card J\times \card J$-minors
of $M_J$ within $k[\XX]$.
Using the Rabinowitsch trick, 
the latter condition can now be tested using Gr\"obner bases machinery.

\subsection{Reminder: generating functions of cones}
\label{ss:genfun}

The following material is largely well-known, see e.g.\ \cite[Ch.\ 13]{Bar08}
and \cite[\S 4.5]{Sta12}. 

\paragraph{Generating functions.}
Given a cone $\cC \subset \RR^n$ and a ring $R$, let $R[\cC_0\cap\ZZ^n]$ be
the \mbox{$R$-subalgebra} of $R[\XX^{\pm 1}]$ spanned by $\XX^\alpha$ with $\alpha\in \cC\cap\ZZ^n$. 
If $\cC \subset \Orth^n$ is a \itemph{rational} cone,
then, within the field of fractions of $\QQ\llb
\XX\rrb$, the series $\sum_{\omega\in \cC\cap \NN_0^n}
\XX^\omega \in \QQ\llb \XX\rrb$ is given by a rational function $\genfun{\cC}$
of the form $\genfun{\cC} = f(\XX)/\prod_{i=1}^r(1-\XX^{\alpha_i})$,
where $f(\XX)\in \ZZ[\cC\cap \ZZ^n]$ 
and $\alpha_1,\dotsc,\alpha_r\in \cC \cap \NN_0^n$.
For an analytic characterisation of $\genfun{\cC}$, 
let $\cC = \cone(\varrho_1,\dotsc,\varrho_e)$ for
$0\not= \varrho_j\in \NN^n_0$.
Then $\ConeRegion(\cC) := \{ \xx\in \Torus^n(\CC) : \abs{\xx^{\varrho_j}} < 1
\text{ for } j=1,\dotsc,e\}$ is a non-empty open set which is independent of
the choice of $\varrho_1,\dotsc,\varrho_e$, and we have
$\genfun{\cC}(\xx) = \sum_{\omega\in\cC\cap \NN_0^n} \xx^{\omega}$ for all
$\xx\in\ConeRegion(\cC)$, the convergence being absolute and compact on $\ConeRegion(\cC)$.

\paragraph{Triangulation.}
The function $\genfun{\cC}$  can be computed in terms of a triangulation of $\cC$.
Here, by a \emph{triangulation} of $\cC$, we mean a rational polyhedral fan $\cF$ in
$\RR^n$ which consists of simplicial cones and whose support is $\cC$.
By the inclusion-exclusion principle, we may write
$\genfun{\cC}$ as a $\ZZ$-linear combination of the rational functions
$\genfun{\sigma}$ for $\sigma\in \cF$; those $\sigma\in \cF$
with $\dim(\sigma) = \dim(\cC)$ have coefficient $1$.
If $\sigma$ is simplicial, say $\sigma = \cone(\beta_1,\dotsc,\beta_d)$ with $d
= \dim(\sigma)$, then $\genfun{\sigma} = (\sum
  \XX^\alpha)/\prod_{i=1}^d(1-\XX^{\beta_i})$, where the summation in the
numerator extends over the lattice points in the half-open parallelepiped 
$\{ \sum_{i=1}^d a_i \beta_i : a_i\in \RR, 0\le a_i < 1\}$.
Using the inclusion-exclusion principle once again, what has been said about
closed cones above extends to (rational) half-open cones $\cC_0\subset
\Orth^n$, see \cite[\S 3.1]{topzeta}; in particular, we obtain a rational function
$\genfun{\cC_0}$ enumerating the lattice points in $\cC_0$. 

\subsection{Local zeta functions associated with regular toric data}
\label{ss:eval_regular}

We record how the machinery developed in \cite{topzeta} 
provides convex-geometric formulae for local zeta functions associated with
regular toric data.

\paragraph{Monomial substitutions.} (Cf.\ \cite[\S 3.2]{topzeta}.)
Let $A\in \Mat_{n\times(m+1)}(\NN_0)$. We assume that the first column of $A$ is $(1,\dotsc,1)^\top$.
Write $A_1,\dotsc,A_n$ for the rows of $A$.
Let $\XX = (X_1,\dotsc,X_n)$ and $\YY = (Y_0,\dotsc,Y_m)$ consist of independent
variables over $\QQ$.
Let $\cA$ be the $\QQ$-subalgebra of $\QQ(\XX)$ generated by $\QQ[\XX]$ and all
$(1-\XX^\alpha)^{-1}$ for $0\not= \alpha\in \NN_0^n$;
similarly, let $\cB$ be the $\QQ$-algebra generated by $\QQ[\YY]$ and all $(1-\YY^\beta)^{-1}$ for
$0\not= \beta\in \NN_0^{m+1}$.
Then $\XX^\alpha \mapsto \YY^{\alpha A}$ extends to a homomorphism $(\blank)^A\colon
\cA \to \cA'$.
In particular, if $\cC_0\subset \Orth^n$ is a rational half-open cone,
then $\genfun{\cC_0}$ belongs to $\cA$ and we may thus consider its image
$\genfun{\cC_0}^A$.
Observe that the rational function $\genfun{\cC_0}^A$ can be evaluated at any
point $(y_0,\dotsc,y_m) \in \CC^{m+1}$ with $0 \le \abs{y_0} < 1$ and $0 \le \abs{y_j}
\le 1$  for $j=1,\dotsc,m$.

\paragraph{Setup.}
Let $(\cC_0;f_1,\dotsc,f_r)$ be a regular toric datum over $k$.
We assume that $\cC_0 \not= \emptyset$.
For $1\le i\le r$, we choose an arbitrary element $\gamma_i\in \supp(\init_{\cC_0}(f_i))$.
Let $J\subset \{ 1,\dotsc,r\}$.
Define $V_J^\circ$ to be the 
subvariety of $\Torus^n_k$ defined by the vanishing of all $\init_{\cC_0}(f_j)$ for $j\in J$
and the non-vanishing of all remaining $\init_{\cC_0}(f_i)$.
For $j\in J$, let $\delta_{jJ}$ be the $j$th standard basis vector of
$\RR^J$ while for $1\le i\le r$, $i\not\in J$, we let $\delta_{iJ} = 0_{\RR^J}$.
Define $\cD_J$ to be the cone consisting of those $(\xi,o) \in \Orth^n \!\times\! \Orth^J$
with $\bil{\gamma_i}\xi + \bil{\delta_{iJ}}o \ge 0$ for $i=1,\dotsc,r$.
Let $\cC_0^J := (\cC_0^{\phantom *}\times\StrictOrth^J) \cap \cD_J$.
Finally, for an $m\times n$ matrix $\beta$ with rows $\beta_1,\dotsc,\beta_m$,
define an $(n+\card J)\times(m+1)$ matrix
$
A_J(\beta) =
 [(1,\dotsc,1)^\top,(\beta_1,0)^\top,\dotsc,(\beta_m,0)^\top].
$

\begin{thm}[{Cf.\ \cite[Thm~4.10]{topzeta}}]
  \label{thm:regular_padic}
  Let $\cT = (\cC_0;f_1,\dotsc,f_r)$ be a regular toric datum over $k$.
  Let $\beta \in \Mat_{m\times n}(\NN_0)$.
  Define $V_J^\circ$, $\cC_0^J$, and $A_J(\beta)$ as above.
  Then for all $p$-adic fields $K\supset k$,
  unless $\fp_K = \fo\cap\fP_K$ belongs to some finite exceptional
  set (depending only on $\cT$), we have
  {
    \small
  \[
  \Zeta_K^{\cT,\beta}(s_1,\dotsc,s_m)
  =
  \!\!\!\!
  \sum_{J\subset\{1,\dotsc,r\}}
  \!\!\!\!
  \noof{\bar V_J^\circ(\fO_K/\fP_K)} \dtimes
  \frac{(q_K-1)^{\card J}}{q_K^n} \dtimes
  \genfun{\cC_0^J}^{A_J(\beta)}
  (q_K^{-1},q_K^{-s_1},\dotsc,q_K^{-s_m})
  \]
  }
  for $s_1,\dotsc,s_m\in \CC$ with $\Real(s_j)\ge 0$, where $\bar\dtimes$
  denotes reduction modulo $\fp_K$.
\end{thm}

\section{Topological zeta functions and regular toric data}
\label{s:topreg}

We describe the function
\func{EvaluateTopologically} (see \S\ref{ss:EvaluateTopologically}) which
computes the topological zeta function
$\Zeta^{\cT,\beta}_{\topo}(\ess_1,\dotsc,\ess_m)$ associated with a regular
toric datum $\cT$ and a matrix $\beta$ as defined in \S\ref{ss:toptcc}.
Our method is based on refined and algorithmic versions of the key ingredients
of \cite[Thm~6.7]{topzeta} applied to Theorem~\ref{thm:regular_padic}, with
further extensions removing the assumptions of ``global non-degeneracy'' from
\cite[\S 6]{topzeta}.

\subsection{Reminder: topological zeta functions via \texorpdfstring{$p$}{p}-adic formulae} 
\label{ss:M}

In \S\ref{ss:assoczeta}, we sketched how formulae of the
form \eqref{eq:basic_denef} can be used to read off the associated topological
zeta function \eqref{eq:basic_top}.
For a rigorous treatment,
we now recall the formalism of \cite{topzeta} which is based on 
work of Denef and Loeser \cite{DL92}.
In particular, we recall the technical conditions regarding 
the $W_i$ in \eqref{eq:basic_top} alluded to above. 

\paragraph{Formal binomial expansions.}
Let $\qq,\tee_1,\dotsc,\tee_m,\ess_1,\dotsc,\ess_m$ be algebraically independent
over $\QQ$;
we regard these variables as symbolic versions of
$q_K^{\phantom{s_1}},q_K^{-s_1},\dotsc,q_K^{-s_m},s_1,\dotsc,s_m$ in Theorem~\ref{thm:regular_padic}. 
Using the binomial series,
we define
$  \qq^{-\ess_j} 
:= \sum_{d=0}^\infty \binom{-\ess_j}{d} (\qq-1)^d
  \in \QQ[\ess_j]\llb \qq-1\rrb$.
Let $W(\qq,\tee_1,\dotsc,\tee_m) \in \QQ(\qq,\tee_1,\dotsc,\tee_m)$
be of the form 
\begin{equation}
  \label{eq:Wshape}
  W(\qq,\tee_1,\dotsc,\tee_m) =
  \frac{f(\qq,\tee_1,\dotsc,\tee_m)}
  {\prod_{i=1}^r (\qq^{a_i}\tee^{b_i}-1)},
\end{equation}
where $f(\qq,\tee_1,\dotsc,\tee_m)\in \QQ[\qq^{\pm 1},\tee_1^{\pm
  1},\dotsc,\tee^{\pm 1}_m]$, $(a_i,b_i)\in \ZZ^{1+m}$, $(a_i,b_i)\not= (0,0)$.
Given $W(\qq,\tee_1,\dotsc,\tee_m)$, 
we obtain $W(\qq,\qq^{-\ess_1},\dotsc,\qq^{-\ess_m}) \in
\QQ(\ess_1,\dotsc,\ess_m)\llp \qq-1\rrp$. 

\clearpage

\begin{notation}
  \label{not:MM}
  \quad
  \begin{enumerate}
    \item
      \label{not:MM1}
      Let $\MM$ be the $\QQ$-subalgebra of $\QQ(\qq,\tee_1,\dotsc,\tee_m)$
      consisting of those $W(\qq,\tee_1,\dotsc,\tee_m)$ of the form \eqref{eq:Wshape}
      with
      $W(\qq,\qq^{-\ess_1},\dotsc,\qq^{-\ess_m})\in \QQ(\ess_1,\dotsc,\ess_m)\llb \qq-1\rrb$.
    \item
      \label{not:MM2}
      Given $W(\qq,\tee_1,\dotsc,\tee_m) \in \MM$,
      we let $\red{W}(\ess_1,\dotsc,\ess_m) \in \QQ(\ess_1,\dotsc,\ess_m)$
      denote the constant term of $W(\qq,\qq^{-\ess_1},\dotsc,\qq^{-\ess_m})$ as
      a series in $\qq-1$.
    \end{enumerate}
\end{notation}

The following generalises topological zeta functions of polynomials
introduced in \cite{DL92}.

\begin{thmdef}[{\cite[\S 5.3]{topzeta}}]
  \label{thmdef:top}
  Let $\Zeta = (\Zeta_K)$ 
  be a family of rational
  functions $\Zeta_K(\tee_1,\dotsc,\tee_m) \in \QQ(\tee_1,\dotsc,\tee_m)$
  indexed by $p$-adic fields $K\supset k$ (up to $k$-isomorphism).
  Suppose that there exists a finite family of $k$-varieties $V_i$ and 
  $W_i(\qq,\tee_1,\dotsc,\tee_m)\in \MM$ ($i\in I$) such that for $p$-adic
  fields $K\supset k$,
  unless  $\fp_K = \fo\cap\fP_K$ belongs to some
  finite 
  set,
  \[
  \Zeta_K(\tee_1,\dotsc,\tee_m) = \sum_{i\in I} \noof{\bar V_i(\fO_K/\fP_K)} \dtimes
  W_i(q_K,\tee_1,\dotsc,\tee_m),
  \]
  where $\bar\dtimes$ denotes reduction modulo $\fp_K$.
  Then the \emph{topological zeta function}
  \[
  \Zeta_{\topo}(\ess_1,\dotsc,\ess_m) :=
  \sum_{i\in I}
  \Euler(V_i(\CC)) \dtimes \red{W_i}(\ess_1,\dotsc,\ess_m) \in \QQ(\ess_1,\dotsc,\ess_m)
  \]
  associated with $\Zeta$ 
  is independent of the choice of $(V_i,W_i(\qq,\tee_1,\dotsc,\tee_m))_{i\in I}$.
\end{thmdef}

When $m = 1$, we write $\ess$ and $\tee$ instead of $\ess_1$ and $\tee_1$.
The formulae for local subalgebra and submodule zeta
functions in \cite{dSG00} give rise to
associated topological zeta functions.

\begin{thmdef}[{\cite[\S 5.4]{topzeta}}]
  \label{thmdef:topsubzeta}
  \quad
  \begin{enumerate}
  \item
    Let $\cA$ be a non-associative $\fo$-algebra whose underlying
    $\fo$-module is free of finite rank $d$.
    For a $p$-adic field $K\supset k$, let
    $\Zeta_{\cA,K}(\tee)\in \QQ(\tee)$  be the rational function
    with $$\Zeta_{\cA,K}(q_K^{-s}) = (1-q_K^{-1})^d \dtimes \zeta_{\cA\otimes_{\fo}\fO_K}(s).$$
    Then $\Zeta_{\cA} := (\Zeta_{\cA,K})$ satisfies the assumptions in
    Theorem\ \&\ Definition~\ref{thmdef:top}.

    The \emph{topological subalgebra zeta function} of $\cA$ is
    $$\zeta_{\cA,\topo}(\ess) := \Zeta_{\cA,\topo}(\ess)\in \QQ(\ess).$$ 
  \item
    Let $M$ be a free $\fo$-module of rank $d$ and let $\cE$ be a
    subalgebra of $\End_{\fo}(M)$.
    For a $p$-adic field $K\supset k$,
    let
    $\Zeta_{\cE\acts M,K}(\tee)\in \QQ(\tee)$ be the rational function with
    $$\Zeta_{\cE\acts M,K}(q_K^{-s}) = (1-q_K^{-1})^d \dtimes \zeta_{(\cE\otimes_{\fo}\fO_K)
    \acts (M\otimes_{\fo}\fO_K)}(s).$$
    Then $\Zeta_{\cE\acts M} := (\Zeta_{\cE \acts M,K})$ satisfies the assumptions in
    Theorem\ \&\ Definition~\ref{thmdef:top}.

    The \emph{topological submodule zeta function} of $\cE$ acting on $M$ is
     $$\zeta_{\cE\acts M,\topo}(\ess) := \Zeta_{\cE \acts M,\topo}(\ess)\in \QQ(\ess).$$ 
  \end{enumerate}
\end{thmdef}

In view of Definition~\ref{d:subzeta}(\ref{d:subzeta2}),
we define the \emph{topological ideal zeta function} of $\cA$ to be
$\zeta_{\cA,\topo}^\normal(\ess) = \zeta_{\Omega(\cA)\acts A,\topo}(\ess)$.

\begin{rem}
  Topological subalgebra zeta functions were first defined in greater generality
  by du Sautoy and Loeser \cite{dSL04}.
  Their definition of the topological subalgebra zeta function of a $\ZZ$-algebra
  $\cA$ of $\ZZ$-rank $d$ coincides with $d!\dtimes \Zeta_{\cA,\topo}(\ess+d)$ in
  our notation;
  we note that the factor $d!$ (a consequence of 
  \cite[Def.\ 7.2]{dSL04} and the remarks following it) seems to be missing from
  the examples in \cite[\S 9]{dSL04}. 
\end{rem}

The following simple observation will be useful for our computations, see \S\ref{s:imp}.

\begin{lemma}
  \label{lem:degbound}
  Notation as in Theorem\ \&\ Definition~\ref{thmdef:topsubzeta},
  the univariate rational functions
  $\zeta_{\cA,\topo}(\ess)\in\QQ(\ess)$ and $\zeta_{\cE\acts
    M,\topo}(\ess)\in\QQ(\ess)$ both have degree $\le 0$ in $\ess$.
\end{lemma}
\begin{proof}
  This follows from the explicit formula \cite[Prop.\ 8.4]{dSL04} for topological zeta functions
  associated with ``cone integrals''---or equivalently, the topological counterpart
  of \cite[Cor.~3.2]{dSG00}.
  Indeed, these formulae express the topological zeta functions under consideration
  as $\ZZ$-linear combinations of rational functions of degree $\le 0$ in $\ess$.
\end{proof}

For all examples of topological subalgebra zeta functions known to
the author, the degree of $\zeta_{\cA,\topo}(\ess)$ is precisely $-d$, where $d$
is the $\fo$-rank of $\cA$, see \cite[\S 8, Conj.~I]{topzeta}.

\begin{thmdef}[{\cite[\S 5.4]{topzeta}; \!cf.\!\!\!\! \cite[\S 6]{dSL04}}]
  Let $G$ be a finitely generated torsion-free nilpotent group.
  Let $\fL(G)$ be the associated Lie $\QQ$-algebra under the Mal'cev correspondence.
  Choose an arbitrary $\ZZ$-subalgebra $\cL \!\subset\! \fL(G)$ 
  which is finitely generated as a $\ZZ$-module
  and which spans $\fL(G)$ over $\QQ$.
  The \emph{topological subgroup zeta function} and \emph{topological normal
    subgroup zeta function} of $G$ are
  $\zeta_{G,\topo}(\ess) := \zeta_{\cL,\topo}(\ess)$ and
  $\zeta^\normal_{G,\topo}(\ess) := \zeta^\normal_{\cL,\topo}(\ess)$, respectively.
  These definitions do not depend on the choice of $\cL$.
\end{thmdef}

\subsection{Topological zeta functions associated with toric data}
\label{ss:toptcc}

Let $\cT$ be a toric datum in $n$ variables over $k$ and
let $\beta$ be an $m\times n$ matrix with entries in $\NN_0$.
The following is a special case of general results in $p$-adic
integration following Denef's fundamental paper \cite{Den87},
cf.\ \cite[Rem.\ 4.7]{topzeta} and \cite[Ex.\ 5.11(vi)]{topzeta}.
Thus, for each $p$-adic field $K\supset k$, the zeta function
$\Zeta^{\cT,\beta}_K(s_1,\dotsc,s_m)$ is rational in
$q_K^{-s_1},\dotsc$, $q_K^{-s_m}$.
We may therefore regard each $\Zeta^{\cT,\beta}_K(s_1,\dotsc,s_m)$ as an element
of $\QQ(\tee_1,\dotsc,\tee_m)$ via $\tee_j \mapsto q_K^{-s_j}$.
After this identification, the collection of rational functions
$\Zeta^{\cT,\beta} := (\Zeta^{\cT,\beta}_K)$ satisfies the assumptions in
Theorem\ \&\ Definition~\ref{thmdef:top}.
Consequently, we obtain a topological zeta
function $\Zeta^{\cT,\beta}_{\topo}(\ess_1,\dotsc,\ess_m)\in
\QQ(\ess_1,\dotsc,\ess_m)$ associated with $\cT$ and $\beta$.
\begin{lemma}
  \label{lem:topadd}
  Let $\cT$ and $\beta$ be as above.
  Let $\mathfrak C$ be a partition of $\cT$
  (see Definition~\ref{d:partition}).
  Then $
  \Zeta^{\cT,\beta}_{\topo}(\ess_1,\dotsc,\ess_m) =
  \sum\limits_{\cT'\in \mathfrak C} \Zeta^{\cT', \beta}_{\topo}(\ess_1,\dotsc,\ess_m)$.
  \qed
\end{lemma}

Hence, if we assume that the invariant (\ref{LOOPINV}) on p.\!~\pageref{LOOPINV}
is preserved by both \func{Simplify} and \func{Reduce} and that
$\func{EvaluateTopologically}(\cT,\beta)$ indeed computes
$\Zeta_{\topo}^{\cT,\beta}(\ess_1,\dotsc,\ess_m)$,
then the correctness of Algorithm~\ref{alg:main} follows from Lemma~\ref{lem:topadd}.

\subsection{Torus factors}
\label{s:torus_factors}
The group $\GL_n(\ZZ)$ admits a natural right-action on $k[\XX^{\pm 1}]$ by
$k$-algebra automorphisms via $(\XX^\alpha)^A :=
\XX^{\alpha A}$ for $\alpha\in \ZZ^n$ and $A\in\GL_n(\ZZ)$.

\begin{lemma}[{\cite[Lem.\ 6.1(i)]{topzeta}}]
  \label{lem:split}
  Let $f_1,\dotsc,f_r\in k[\XX^{\pm 1}]$ be non-zero Laurent polynomials.
  Let $\cN = \Newton(f_1\dotsb f_r)$ and $d = \dim(\cN)$.
  For $1\le i\le r$, choose $\alpha_i\in \supp(f_i)$.
  Then there exists $A\in \GL_n(\ZZ)$
  such that $(\XX^{-\alpha_i}f_i)^A \in k[X_1^{\pm 1},\dotsc,X_d^{\pm 1}]$ for
  $1\le i\le r$.
\end{lemma}

Recall that Theorem~\ref{thm:regular_padic} featured certain explicitly defined
subvarieties of algebraic tori over $k$. 
The relevance of Lemma~\ref{lem:split} is due to the following geometric consequence.
\begin{cor}
  \label{cor:split}
  Write $g_i := (\XX^{-\alpha_i}f_i)^A$ for $1\le i\le r$.
  Let $V$ be the subvariety of $\Torus^n_k$ defined by $f_1 = \dotso = f_r = 0$ and
  let $U$ be the subvariety of $\Torus^d_k$ defined by $g_1 = \dotso = g_r = 0$.
  Then $V \approx_k U\times_{\Spec(k)} \Torus_k^{n-d}$. \qed
\end{cor}

The proof of Lemma~\ref{lem:split} given in \cite{topzeta} easily translates
into an algorithm.
Indeed, let $M$ be the $\ZZ$-submodule of $\ZZ^n$ generated by
$\bigcup_{i=1}^r\supp(\XX^{-\alpha_i} f_i)$.
Then $M$ has rank $d$, see the proof of \cite[Lem.~6.1(i)]{topzeta}.
Let $B$ be any matrix over $\ZZ$ (of size $e\times n$, say) whose rows span $M$
over $\ZZ$.
We may find $C\in \GL_e(\ZZ)$ and $A\in
\GL_n(\ZZ)$ such that $CBA$ is in Smith normal form. 
Evidently, $A$ then satisfies the desired conditions in Lemma~\ref{lem:split}.

\subsection{Rewriting Theorem~\ref{thm:regular_padic}}
\label{ss:rewrite}

In \cite{topzeta}, we used Lemma~\ref{lem:split}
to rewrite the explicit formulae in \cite[Thm~4.10]{topzeta} in a shape
compatible with Theorem \& Definition~\ref{thmdef:top}. 
In order to be able to explicitly compute associated topological zeta functions,
we now consider an algorithmic version of this rewriting process applied to the
formula in Theorem~\ref{thm:regular_padic}.

Let $\cT = (\cC_0;f_1,\dotsc,f_r)$ be a non-trivial regular
toric datum.
Let $\beta$, $V_J^\circ$, $\cC_0^J$, and $A_J(\beta)$ be as in
Theorem~\ref{thm:regular_padic}.
Let $\cN := \Newton(f_1\dotsb f_r)$.
Since $\cT$ is balanced, by Lemma~\ref{lem:balanced_crit}, 
there exists a (unique) face $\tau\subset \cN$ such that $\cC_0\subset
\NormalCone_\tau(\cN)$;
in particular, $\dim(\cC_0) \le n - \dim(\tau)$.
Let $\tau = \tau_1 + \dotsb + \tau_r$ be the decomposition of $\tau$ into faces $\tau_i
\subset \Newton(f_i)$; hence, $\tau_i = \Newton(\init_{\cC_0}(f_i))$
(cf.\ \cite[Lem.\ 6.1(iii)]{topzeta}).

For $J\subset \{1,\dotsc,r\}$, let $V_J$ be the subvariety of $\Torus^n_k$
defined by $\init_{\cC_0}(f_j) = 0$ for all $j\in J$.
Note that using the inclusion-exclusion principle, we may replace 
$\noof{\bar V_J^\circ(\fO_K/\fP_K)}$ in Theorem~\ref{thm:regular_padic} by
$\sum_{J\subset T \subset \{1,\dotsc,r\}} (-1)^{\card T + \card J} \dtimes
\noof{\bar V_T^{\phantom\circ}(\fO_K/\fP_K)}$.

Again, let $J\subset \{1,\dotsc,r\}$. Write $d(J) := \dim\bigl(\sum_{j\in J}\tau_j\bigr)$.
Equivalently, $d(J)$ is the dimension of $\sum_{j\in
  J}\Newton(\init_{\cC_0}(f_j)) = \Newton\bigl(\prod_{j\in J}\init_{\cC_0}(f_j)\bigr)$.
By Lemma~\ref{lem:split},
we may thus construct $B_J\in \GL_n(\ZZ)$ and non-zero $g_j \in k[X_1^{\pm 1},\dotsc,X_{d(J)}^{\pm 1}]$
such that $g_j^{-1}\init_{\cC_0}(f_j)^{B_J}$ is a Laurent monomial for each $j\in J$.
Let $U_J$ be the subvariety of $\Torus^{d(J)}_k$ defined by $g_j = 0$ for all
$j\in J$
so that $V_J \approx_k U_J \times_{\Spec(k)} \Torus_k^{n-d(J)}$ (see Corollary~\ref{cor:split}).
Finally, define
{\small
\begin{equation}
  \label{eq:W}
  W_J(\qq,\tee_1,\dotsc,\tee_m) := 
  \qq^{-n}(\qq-1)^{n-\dim(\tau)+\card J} \dtimes
  \genfun{\cC_0^J}^{A_J(\beta)}(\qq^{-1},\tee_1,\dotsc,\tee_m)
  \in \QQ(\qq,\tee_1,\dotsc,\tee_m).
\end{equation}
}
\begin{prop}
  \label{prop:regular_padic_rewritten}
  Notation as above; in particular, let $\cT = (\cC_0;f_1,\dotsc,f_r)$ be regular.
  For all $p$-adic fields $K\supset k$, unless $\fp = \fo\cap\fP_K$
  belongs to some finite exceptional set, 
  {\footnotesize
    \begin{equation}
      \label{eq:regular_padic_rewritten}
      \Zeta^{\cT{\!, }\beta}_K (s_1,\dotsc,s_m) =
      \sum_J
      \Bigl(
      \sum_{J\subset T}
      (-1)^{\card J + \card T} \dtimes
      \noof{\bar U_T(\fO_K/\fP_K)}
      \dtimes (q_K-1)^{\dim(\tau)-d(T)}
      \Bigr)
      \dtimes
      W_J(q_K^{\phantom{s_1}}, q_K^{-s_1},\dotsc,q_K^{-s_m}),
    \end{equation}
  }
  where $J$ and $T$ range over subsets of $\{1,\dotsc,r\}$
  and $s_1,\dotsc,s_m\in \CC$ with $\Real(s_j) \ge 0$. 
  \qed
\end{prop}

The two crucial features of Proposition~\ref{prop:regular_padic_rewritten}
compared with Theorem~\ref{thm:regular_padic} are
\begin{enumerate}
\item
  the $U_T$ are embedded as {closed} (instead of locally closed)
  subvarieties of tori and
\item
  each $W_J(\qq,\tee_1,\dotsc,\tee_m)$ belongs to the algebra $\MM$
  from \S\ref{ss:M} (Corollary~\ref{cor:WJinM}).
\end{enumerate}

\subsection{Computing formal reductions modulo \texorpdfstring{$\qq-1$}{q-1}}
\label{ss:compute_reduction}

We show that each $W_J$ in Proposition~\ref{prop:regular_padic_rewritten}
belongs to the $\QQ$-algebra $\MM$ from \S\ref{ss:M}.
While this statement alone is merely a special case of
\cite[Lem.\ 6.9(i)]{topzeta}, the proof given here provides
an algorithm for computing $\red{W}(\ess_1,\dotsc,\ess_m)$ (see Notation~\ref{not:MM}(\ref{not:MM2})).

\begin{lemma}
  \label{lem:compute_reduction}
  Let $\cB_0\subset \Orth^r$ be a non-empty rational half-open cone of dimension
  $d$.
  Let $A$ be an $r\times(m+1)$-matrix with entries in $\NN_0$ and suppose that
  the first column of $A$ is $(1,\dotsc,1)^\top$.
  Then $W(\qq,\tee_1,\dotsc,\tee_m) := (\qq-1)^d
  \genfun{\cB_0}^A(\qq^{-1},\tee_1,\dotsc,\tee_m)$ belongs to the algebra $\MM$
  from \S\ref{ss:M}.
  Moreover,
  $-d \le \deg_{\ess_j}\bigl(\red{W}(\ess_1,\dotsc,\ess_m)\bigr) \le 0$ for $j = 1,\dotsc, m$.
\end{lemma}
\begin{proof}
  Let $\bm\lambda = (\lambda_1,\dotsc,\lambda_r)$ consist of independent
  variables over $\QQ$.
  We regard generating functions of rational half-open cones in $\Orth^r$ as
  elements of $\QQ(\bm\lambda)$, cf.\ \S\ref{ss:genfun}.

  Let $\cF$ be a triangulation of the closure $\bar\cB_0$ of $\cB_0$ into simplicial cones.
  Let $\sigma \in \cF$, say $\sigma = \cone(\varrho_1,\dotsc,\varrho_e)$, where
  $\varrho_1,\dotsc,\varrho_e\in \NN_0^r$ are primitive vectors and $e =
  \dim(\sigma)$.
  Let $\Pi(\sigma) = \{ a_1 \varrho_1 + \dotsb + a_e \varrho_e : 0 \le a_i < 1\}$.
  Then $\genfun{\sigma} = \bigl(\sum_{\beta\in \Pi(\sigma)\cap \ZZ^r}
  \bm\lambda^\beta\bigr)/\prod_{i=1}^e(1-\bm\lambda^{\varrho_i})$.
  Write $\one = (1,\dotsc,1)$ and $A = [\one^\top,\alpha_1^\top,\dotsc,\alpha_m^\top]$.
  Therefore
  \[
  \genfun{\sigma}^A(\qq^{-1},\tee_1,\dotsc,\tee_m) =
  \frac{\sum\limits_{\beta\in \Pi(\sigma)\cap\ZZ^r} \qq_{\phantom 1}^{-\bil\one\beta}
    \tee_1^{\bil{\alpha_1}\beta} \dotsb  \tee_m^{\bil{\alpha_m}{\beta}}}
  {\prod\limits_{i=1}^e\Bigl( 1- \qq_{\phantom 1}^{-\bil\one{\varrho_i}}
    \tee_1^{\bil{\alpha_1}{\varrho_i}} \dotsb 
    \tee_m^{\bil{\alpha_m}{\varrho_i}} \Bigr)
    }.
  \]
  For $b\in \NN$ and $\bm a = (a_1,\dotsc,a_m) \in \NN_0^m$, let
  $W_{\bm a,b}(\qq,\tee_1,\dotsc,\tee_m) :=
  \frac{\qq-1}{1-\qq_{\phantom 1}^{-b}\tee_1^{a_1}\dotsb \tee_m^{a_m}}$.
  Then $W_{\bm a, b}(\qq,\tee_1,\dotsc,\tee_m) \in \MM$ and
  $\red{W_{\bm a,b}}(\ess_1,\dotsc,\ess_m) = 1/(a_1 \ess_1 + \dotsb + a_m \ess_m + b)$.
  We conclude that $Z_{\sigma}(\qq,\tee_1,\dotsc,\tee_m) := (\qq-1)^d \dtimes
  \genfun{\sigma}^A(\qq^{-1},\tee_1,\dotsc,\tee_m)$ belongs to $\MM$.
  Moreover, if $\dim(\sigma) = e < d = \dim(\cB_0)$, then
  $\red{Z_\sigma}(\ess_1,\dotsc,\ess_m) = 0$.
  If, on the other hand, $d = e$, then
  \begin{equation}
    \label{eq:Zsigma}
    \red{Z_\sigma}(\ess_1,\dotsc,\ess_m) =
    \frac{\noof{\Pi(\sigma)}}{\prod_{i=1}^e \bil{\varrho_i A}{(1,\ess_1,\dotsc,\ess_m)}};
  \end{equation}
  we note that $\noof{\Pi(\sigma)}$ is the usual multiplicity of
  the simplicial cone $\sigma$, see \cite[Prop.\ 11.1.8]{CLS11}.

  As a consequence of the inclusion-exclusion principle, we may write
  $
  W(\qq,\tee_1,\dotsc,\tee_m) = \sum_{\sigma\in \cF} c_\sigma Z_\sigma(\qq,\tee_1,\dotsc,\tee_m)
  $,
  where $c_\sigma\in \ZZ$ and 
  $c_\sigma = 1$ whenever $\dim(\sigma) = d$.
  In particular, $W(\qq,\tee_1,\dotsc,\tee_m)\in \MM$ and
  $\red{W}(\ess_1,\dotsc,\ess_m) = \sum_\sigma
  \red{Z_\sigma}(\ess_1,\dotsc,\ess_m)$,
  where the sum is taken over those $\sigma \in \cF$ with $\dim(\sigma) = d$
  only.
  In the expression for $\red{Z_\sigma}(\ess_1,\dotsc,\ess_m)$ given in
  \eqref{eq:Zsigma},
  the numerator is a positive constant and the denominator is a product of
  $\dim(\sigma)$ factors of the form $a_1 \ess_1 + \dotsb + a_m \ess_m + b$
  for $a_1,\dotsc,a_m\in \NN_0$ and $b\in \NN$.
  By the non-negativity of all these numbers,
  the degree of $\red{W}(\ess_1,\dotsc,\ess_m)$ in $\ess_j$ is simply
  the maximal degree of any $\red{Z_\sigma}(\ess_1,\dotsc,\ess_m)$ in $\ess_j$ for
  $\sigma\in \cF$ with $\dim(\sigma) = d$. 
\end{proof}

\begin{rem}
  \quad
  \begin{enumerate}
  \item
    In \cite[\S 5]{DL92}, Denef and Loeser gave an explicit convex-geometric
    formula for the topological zeta function associated with a suitably
    non-degenerate polynomial.
    In view of the $p$-adic formulae of Denef and Hoornaert \cite{DH01},
    the explicit descriptions of the rational functions $J(\tau,s)\in \QQ(s)$ in
    terms of triangulations in \cite{DL92} can be regarded as a special case of
    Lemma~\ref{lem:compute_reduction}. 
  \item
    The proof of Lemma~\ref{lem:compute_reduction} shows that
    the rational function $\red{W}(\ess_1,\dotsc,\ess_m)$ only depends on the
    closure of $\cB_0$.
  \end{enumerate}
\end{rem}

\begin{cor}
  \label{cor:WJinM}
  Notation as in \S\ref{ss:rewrite}.
  For each $J\subset \{1,\dotsc,r\}$, the rational function
  $W_J(\qq,\tee_1,\dotsc,\tee_m)$ belongs to $\MM$.
  Moreover, $\red{W_J}(\ess_1,\dotsc,\ess_m) = 0$ if and only if $\dim(\cC_0^J)
  < n - \dim(\tau) + \card{J}$.
\end{cor}
\begin{proof}
  $\dim(\cC_0) \le n - \dim(\tau)$ 
  and $\cC_0^J \subset \cC_0 \times \StrictOrth^J$ whence $\dim(\cC_0^J) \le n - \dim(\tau) +
  \card J$.
\end{proof}

\subsection{Computing Euler characteristics}
\label{ss:euler}

Let $U$ be the closed subvariety of $\Torus^n_k$ defined by $f_1 = \dotsb = f_r =
0$ for $f_1,\dotsc,f_r\in k[\XX^{\pm 1}]$.
The typical example to bear in mind is the case where
$f_1,\dotsc,f_r$  are the initial forms of a non-empty regular toric datum.
We now consider the computation of the topological Euler characteristic $\Euler(U(\CC))$.
The function \func{EvaluateTopologically} (see \S\ref{ss:EvaluateTopologically})
will rely on our ability to compute these numbers.

\paragraph{General methods.}
Aluffi \cite{Alu03} described an algorithm for computing the topological Euler characteristic
of a not necessarily smooth projective variety in characteristic zero based 
on the computation of so-called Chern-Schwartz-MacPherson classes;
for recent developments, see \cite{Jos13,Hel14}.
In principle, such general algorithms can be used to compute $\Euler(U(\CC))$
from above.
Indeed, after clearing denominators, we may assume that $f_1,\dotsc,f_r\in
k[\XX]$. Let $\tilde f_i\in k[X_0,\dotsc,X_{n+1}]$ denote the homogenisation of
$f_i$. Let the subvarieties $V,W \subset \PP^n_k$ be defined by $\tilde f_1 = \dotsb =
\tilde f_r = 0$ and $\tilde f_1 = \dotsb = \tilde f_r = X_0 \dotsb X_n = 0$, respectively.
Then $U \approx_k V\setminus W$ and so $\Euler(U(\CC)) = \Euler(V(\CC)) - \Euler(W(\CC))$;
cf.\ \cite[\S 2.8]{Alu03}.

In practice, while implementations of \cite{Alu03,Jos13,Hel14} exist, these
methods are usually too costly for our applications to the computation
of topological zeta functions.
For example, the computation of $\zeta_{\Fil_4,\topo}(\ess)$ previously announced
in \cite[\S 7.3]{topzeta} involves the Euler characteristics of thousands of
subvarieties of $\Torus^{15}_{\QQ}$.
In our implementation (see \S\ref{s:imp}), we therefore attempt to compute
Euler characteristics using special-purpose methods.

\paragraph{The Bernstein-Khovanskii-Kushnirenko Theorem.}
As we already exploited in \cite[\S 6]{topzeta},
if $(f_1,\dotsc,f_r)$ is non-degenerate in the sense of Khovanskii \cite[\S 2]{Kho77},
then \cite[\S 3, Thm~2]{Kho77} provides an explicit formula for
$\Euler(U(\CC))$ in terms of various mixed volumes associated with the Newton polytopes of
$f_1,\dotsc,f_r$.
Since $U$, and hence $\Euler(U(\CC))$, only depends on the radical of the ideal
generated by $f_1,\dotsc,f_r$ within $k[\XX^{\pm 1}]$,
when $(f_1,\dotsc,f_r)$ is degenerate in Khovanskii's sense,
we can try to use standard techniques such as multivariate polynomial division
(after clearing denominators) and saturation to ``simplify'' $(f_1,\dotsc,f_r)$,
e.g.\ by reducing $\sum_{i=1}^r \card{\supp(f_i)}$.

\paragraph{Decomposing subvarieties of tori.}
If  $(f_1,\dotsc,f_r)$ remains degenerate after applying the simplification
steps indicated above, we try to decompose $U$ as follows.
Suppose that after renumbering of $1,\dotsc,r$ (or, more generally, a suitable
application of a matrix from $\GL_n(\ZZ)$ as in \S\ref{s:torus_factors}) and
rescaling of $f_1$ by Laurent monomials (which does not change $U$), we have
$f_1 = X_n - w$ for $w \in k[X_1^{\pm 1},\dotsc,X_{n-1}^{\pm 1}]$.
Let $V \subset \Torus^{n-1}_k$ be the subvariety
defined by $f_2(X_1,\dotsc,X_{n-1},w) = \dotsb = f_r(X_1,\dotsc,X_{n-1},w) = 0$ and
let
$W\subset V$ be defined by $w = 0$. Then $U \approx_k V\setminus W$ and we can
recursively try to compute $\Euler(U(\CC)) = \Euler(V(\CC)) - \Euler(W(\CC))$
using the techniques mentioned above.

\paragraph{}
In practice, combining these methods  often suffices to compute Euler characteristics in
Algorithm~\ref{alg:main}.

\subsection{An algorithm for computing topological zeta functions associated with regular toric data}
\label{ss:EvaluateTopologically}

The following is a topological version of Proposition~\ref{prop:regular_padic_rewritten}.

\begin{prop}
  \label{prop:regular_top_rewritten}
  Notation as in Proposition~\ref{prop:regular_padic_rewritten}; in particular,
  $\cT = (\cC_0;f_1,\dotsc,f_r)$ is a regular toric datum. Then
  {\small
    \[
    \Zeta^{\cT,\beta}_{\topo}(\ess_1,\dotsc,\ess_m) = 
    \sum_{\substack{J \subset T \subset \{1,\dotsc,r\},\\
        n - d(T) + \card J = \dim(\cC_0^J)
      }}
    \!\!\!\! \!\!\!\!
    (-1)^{\card J + \card T} \dtimes
    \Euler(U_T(\CC)) \dtimes
    \red{W_J}(\ess_1,\dotsc,\ess_m).
    \]
  }
\end{prop}
\begin{proof}
  Let $J\subset T$.
  Proposition~\ref{prop:regular_padic_rewritten} and Corollary~\ref{cor:WJinM}
  show that $\Zeta^{\cT,\beta}_{\topo}(\ess_1,\dotsc,\ess_m)$ is the sum of
  $(-1)^{\card J + \card T} \dtimes  \Euler(U_T(\CC)) \dtimes
  \red{W_J}(\ess_1,\dotsc,\ess_m)$
  over pairs $J\subset T$ with $d(T)=\dim(\tau)$ and $n-\dim(\tau) + \card{J}
  = \dim(\cC_0^J)$.
  The latter two conditions are both satisfied if and only if $n-d(T) + \card J = \dim(\cC_0^J)$ since
  $d(T)\le \dim(\tau)$ and $n-\dim(\tau) + \card J \ge \dim(\cC_0^J)$.
\end{proof}

The Euler characteristics $\Euler(U_T(\CC))$ can be determined as in
\S\ref{ss:euler},
while the rational functions $\red{W_J}(\ess_1,\dotsc,\ess_m)$ may be computed as explained in \S\ref{ss:compute_reduction}.
We obtain the following algorithm. 

\paragraph{The function \textrm{\func{EvaluateTopologically}}.}
We let \func{EvaluateTopologically} denote the function which, given a regular
toric datum $\cT$ in $n$ variables over $k$ and a matrix
$\beta\in\Mat_{m\times n}(\NN_0)$, computes
$\Zeta^{\cT,\beta}_{\topo}(\ess_1,\dotsc,\ess_m)$ using
Proposition~\ref{prop:regular_top_rewritten}.
Specifically, for each $J\subset \{1,\dotsc,r\}$, we
use \S\ref{s:torus_factors} and \S\ref{ss:euler} to first compute $e_J :=
     \sum
     (-1)^{\card J + \card T} \dtimes
     \Euler(U_T(\CC))$,
the sum being over those $T \supset J$ with $n - d(T) + \card J = \dim(\cC_0^J)$
as in Proposition~\ref{prop:regular_top_rewritten}.
Only if $e_J$ turns out to be non-zero, do we proceed to compute
 $\red{W_J}(\ess_1,\dotsc,\ess_m)$ using
a triangulation of the closure of $\cC_0^J$ as in the proof of Lemma~\ref{lem:compute_reduction}.

\begin{rem}
  \label{r:ET}
  \quad
  \begin{enumerate}
  \item
  \label{r:ET1}
  The topological zeta functions that we seek to compute can be written as
  univariate specialisations of topological zeta functions associated
  with toric data, see Remark~\ref{r:special}.
  In practice, we avoid the costly multivariate rational function arithmetic
  altogether and apply these specialisations directly in
  the triangulation step of \func{EvaluateTopologically};
  for a theoretical justification, use \cite[Rem.\ 5.15]{topzeta}.
  \item
  \label{r:ET2}
  If we ignore the simplification step in line~\ref{alg:simplification_step} of
  Algorithm~\ref{alg:main}, then, at this point, we have obtained an algorithmic
  version of \cite[Thm~6.7]{topzeta} (restricted to the integrals considered
  here).
  Namely, let $(\cC_0;\ff)$ be a toric datum as in
  Algorithm~\ref{alg:main} and
  suppose that $\ff$ is globally non-degenerate in the sense of \cite[Def.\ 4.2(ii)]{topzeta}.
  It follows from \cite[Lem.\ 6.1]{topzeta} that
  each $\cT \in \func{Balance}(\cC_0;\ff)$ is regular and that the
  defining polynomials of the varieties $U_T$ in
  $\func{EvaluateTopologically}(\cT,\beta)$
  satisfy Khovanskii's non-degeneracy conditions;
  the computation of $e_J$ as part of \func{EvaluateTopologically} is then a direct
  implementation of \cite[Prop.\ 6.5]{topzeta}.
  \end{enumerate}
\end{rem}

\section{Simplification and reduction}
\label{s:simplify}

We now describe the remaining two functions \func{Simplify} and
\func{Reduce} in Algorithm~\ref{alg:main}.

\subsection{Weak and strong equivalence of toric data}

\begin{defn}
    Let $(\cC_0;\ff)$ and $(\cD_0;\bm g)$ be toric data over $k$.
  \begin{enumerate}
  \item
    We say that $(\cC_0;\ff)$ and $(\cD_0;\bm g)$ are \emph{strongly equivalent} if
    $\cC_0 = \cD_0$ and
    there exists a finite $S\subset \Spec(\fo)$ such that if $K\supset k$ is a
    $p$-adic field with $\fo\cap\fP_K\not\in S$, then 
    $\norm{\ff(\xx)}_K = \norm{\bm g(\xx)}_K$ for all $\xx\in
    \Torus^n(K)$ with $\nu_K(\xx)\in \cC_0$.
  \item
    We say that $(\cC_0;\ff)$ and $(\cD_0;\bm g)$ are
    \emph{weakly equivalent} if there is a finite $S\subset \Spec(\fo)$ such
    that $(\cC_0;\ff)_K = (\cD_0;\bm g)_K$ for all
    $p$-adic fields $K\supset k$ with $\fo\cap\fP_K\not\in S$.
  \end{enumerate}
\end{defn}

Strong equivalence implies weak one but
the converse is false; for example, 
$\bigl(\Orth;X_1^{-1}\bigr)$ and $\bigl(\Orth;X_1^{-1}-X_1^{\phantom 1}\bigr)$
are weakly equivalent but not strongly so.
Theorem\ \&\ Definition~\ref{thmdef:top} yields the following.
\begin{lemma}
  Let $\cT$ and $\cT'$ be weakly equivalent toric data in $n$
  variables over $K$ and let $\beta\in \Mat_{m\times n}(\NN_0)$.
  Then $\Zeta^{\cT,\beta}_{\topo}(\ess_1,\dotsc,\ess_m) = \Zeta^{\cT',
    \beta}_{\topo}(\ess_1,\dotsc,\ess_m)$.
  \qed
\end{lemma}

We now collect some instances of these equivalences
in a form that resembles Gaussian elimination and the multivariate polynomial
division algorithm (see e.g.\ \cite[\S 1.5]{AL94}).
By a \emph{term}, we mean a Laurent polynomial of the form $c \XX^\alpha$,
where $c\in k^\times$ and $\alpha\in \ZZ^n$. 
Given a rational half-open cone $\cC_0 \subset \Orth^n$, its
dual $\cC_0^* = \{ \omega\in \RR^n : \bil\alpha\omega \ge 0 \text{ for all }
\alpha\in \cC_0\}$ is a rational \itemph{closed} cone which contains $\Orth^n$.
The unit group $k[\cC_0^*\cap \ZZ^n]^\times$ of $k[\cC_0^*\cap \ZZ^n]$ (see \S\ref{ss:genfun})
 consists precisely of those terms $c\XX^\alpha$ with $c\in k^\times$ and
 $\alpha\in \cC_0^\perp \cap \ZZ^n$.
In the following, we assume that $r\in \NN_0$ is large enough for the statements given to make sense.

\begin{lemma}
  \label{lem:eqop}
  Let $(\cC_0;\ff) = (\cC_0; f_1,\dotsc,f_r)$ be a toric datum over $k$.
  We let $\sim_s$ and $\sim_w$ signify strong and weak equivalence,
  respectively.
  Then:
  \begin{enumerate}
  \item[(\textlabel{S0}{S0})]
    $(\cC_0;f_1,\dotsc,f_r,0) \sim_s (\cC_0;f_1,\dotsc,f_r)$.
  \item[(\textlabel{S1}{S1})]
    $(\cC_0;f_1,\dotsc,f_r) \!\sim_s\! (\cC_0;f_{1\sigma},\dotsc,f_{r\sigma})$ 
    for any permutation $\sigma\in \mathrm{Sym}(r)$.
  \item[(\textlabel{S2}{S2})]
    If $u\in k[\cC_0^*\cap \ZZ^n]^\times$ and $v\in k[\cC_0^*\cap \ZZ^n]$, then
    $(\cC_0; \ff) \sim_s (\cC_0; uf_1 + vf_2,f_2,\dotsc,f_r)$.
  \item[(\textlabel{S3}{S3})]
    If $f_1$ is $\cC_0$-balanced
    and $\init_{\cC_0}(f_1)$ is
    a term, then $(\cC_0; \ff) \sim_s (\cC_0;\init_{\cC_0}(f_1),f_2,\dotsc,f_r)$.
  \item[(\textlabel{W1}{W1})]
    If $v\in k[\cC_0^*\cap \ZZ^n]$,
    then $(\cC_0;\ff) \sim_w (\cC_0;f_1 + v,f_2,\dotsc,f_r)$.
  \item[(\textlabel{W2}{W2})]
    If $f_1= c\XX^\alpha$ for $c\in k^\times$ and $\alpha\in \ZZ^n$,
    then $(\cC_0;\ff)\sim_w  (\cC_0\cap \{\alpha\}^*;f_2,\dotsc,f_r)$.
  \end{enumerate}
\end{lemma}
\begin{proof}
  (\ref{S0}) and (\ref{S1}) are obvious.
  Let $K\supset k$ be a $p$-adic field. 
  Let $\omega\in \ZZ^n$ and $\xx \in \Torus^n(K)$ with $\nu_K(\xx) = \omega$.
  Write $\xx = (\pi_K^{\omega_1}u_1,\dotsc,\pi_K^{\omega_n}u_n)$ for $\uu\in\Torus^n(\fO_K)$.
  Then, for any non-zero $g\in \fO_K[\XX^{\pm 1}]$, we have
  $g(\xx) = \pi_K^{\bil\alpha\omega} \dtimes (\init_\omega(g)(\uu) + \cO(\pi_K))$,
  where $\alpha\in \supp(\init_\omega(g))$ is arbitrary.
  Hence, if $u\in k[\cC_0^*\cap\ZZ^n]^\times$, $v\in k[\cC_0^*\cap\ZZ^n]$,
  and $\xx\in \Torus^n(K)$ with $\nu_K(\xx)\in \cC_0$, then $\abs{u(\xx)}_K = 1$ and
  $\abs{v(\xx)}_K\le 1$, provided that the 
  unique non-zero coefficient of $u$ is a $\fP_K$-adic unit and
  all coefficients of $v$ are $\fP_K$-adic integers.
  (\ref{W1}) is now obvious.
  (\ref{S2}) follows since if $a,b\in K$ and $e\in \fO_K$,
  then $\norm{a,b}_K = \norm{a+eb,b}_K$.
  Indeed, if $\abs{a}_K \le \abs{eb}_K$, then $\abs{a}_K,\abs{a+eb}_K \le \abs{b}_K$;
  if, on the other hand, $\abs{a}_K > \abs{eb}_K$, then $\abs{a+eb}_K = \abs{a}_K$.
  For (\ref{S3}) and (\ref{W2}), let $f_1$ be $\cC_0$-balanced with
  $\init_{\cC_0}(f_1) = c\XX^{\alpha}$, where $c\in k^\times$ and $\alpha\in
  \ZZ^n$; we may assume that $c\in \fO_K^\times$ and $f_1\in \fO_K[\XX^{\pm 1}]$.
  Then for $x\in \Torus^n(K)$ with $\nu_K(\xx) = \omega \in \cC_0\cap \ZZ^n$, we
  have $\abs{f_1(\xx)}_K = q_K^{-\bil\alpha\omega} = \abs{c \xx^{\alpha}}_K$.
  Hence, $\abs{f_1(\xx)}_K \le 1$ if and only if $\bil\alpha\omega\ge 0$.
\end{proof}

\subsection{Simplification}
\label{ss:simplify}

What we call simplification is the systematic application of Lemma~\ref{lem:eqop} (with the
exception of operation (\ref{S2}), see Proposition~\ref{prop:simplify_preserves}) to
toric data.
\begin{defn}
  \label{d:simple}
  A toric datum $(\cC_0;f_1,\dotsc,f_r)$ is
  \emph{simple} if the following conditions are satisfied for $i = 1,\dotsc,r$:
  \begin{enumerate}
  \item 
    $f_i\not= 0$ and no term of $f_i$ lies in $k[\cC_0^*\cap \ZZ^n]$.
  \item
    If $f_i^{\phantom 1}\! f_j^{-1} \in k[\cC_0^*\cap\ZZ^n]$ for $1\le j\le
    r$, then $i = j$.
  \item
    If $f_i$ is $\cC_0$-balanced, then $\init_{\cC_0}(f_i)$ consists of at least
    two terms.
  \end{enumerate}
\end{defn}

\paragraph{The function \textrm{\func{Simplify}}.}
We now describe the function \func{Simplify} in
Algorithm~\ref{alg:main}.
Given a toric datum $(\cC_0;f_1,\dotsc,f_r)$, we remove those terms of each
$f_i$ that lie in $k[\cC_0^*\cap\ZZ^n]$. 
We then discard those $f_j$ with $f_j = 0$ altogether.
Next, if $f_i^{\phantom 1}\!f_j^{-1} \in k[\cC_0^*\cap\ZZ^n]$ for $i\not= j$,
then we discard $f_i$.
Finally, if some $f_i$ is $\cC_0$-balanced with $\init_{\cC_0}(f_i) = c\XX^\alpha$
(where $c\in k^\times$), then we discard $f_i$ and replace $\cC_0^{\phantom *}$ by $\cC_0 \cap
\{\alpha\}^*$.
Since shrinking $\cC_0^{\phantom *}$ enlarges its dual $\cC_0^*$, further terms
might now become redundant.
We therefore repeatedly apply the above
process until $(\cC_0;f_1,\dotsc,f_r)$ stabilises.
As each non-trivial operation decreases $r + \sum_{i=1}^r \noof{\supp(f_i)}$, after
finitely many steps, we obtain a simple toric datum
which is weakly equivalent to the original $(\cC_0;f_1,\dotsc,f_r)$ by Lemma~\ref{lem:eqop}.

\paragraph{}
The reason we only made very limited use of Lemma~\ref{lem:eqop}(\ref{S2}) is 
to ensure the following:
\begin{prop}
  \label{prop:simplify_preserves}
  Let $\cT$ be a toric datum.
  \begin{enumerate}
  \item
  \label{prop:simplify_preserves1}
    If $\cT$ is balanced, then so is $\func{Simplify}(\cT)$.
  \item
  \label{prop:simplify_preserves2}
    If $\cT$ is regular, then so is $\func{Simplify}(\cT)$.
  \end{enumerate}
\end{prop}
\begin{proof}
  Let $\cT = (\cC_0;\ff)$ be non-trivial.
  The properties of being balanced or regular are preserved if we
  discard polynomials, shrink $\cC_0$, or remove non-initial terms.
  Let $f\in \ff$ be $\cC_0$-balanced
  and suppose that $\init_{\cC_0}(f)$ contains a term $c\XX^\alpha$ with $c\in
  k^\times$ and $\alpha\in \cC_0^*$. 
  Then $0 \le \bil\alpha\omega \le \bil\beta\omega$ for all $\beta\in \supp(f)$ and
  $\omega\in \cC_0$ whence $\func{Simplify}(\cT)$ will discard $f$ entirely.
\end{proof}

\subsection{Reduction}
\label{ss:reduce}

We now describe the ``reduction step'' in Algorithm~\ref{alg:main}.
The function \func{Reduce} takes as input a balanced (Definition~\ref{d:balanced}(\ref{d:balanced2}))
and simple (Definition~\ref{d:simple}) toric datum $\cT =
(\cC_0;f_1,\dotsc,f_r)$ which is singular (Definition~\ref{d:regular}).
We therefore cannot directly use \func{EvaluateTopologically} from
\S\ref{ss:EvaluateTopologically}
to compute the associated topological zeta function.
Our goal is to construct and return a partition
$\mathfrak C$ of $\cT$.
Ideally, we would like $\mathfrak C$ to consist of regular toric data
but our immediate goal is more modest:
we systematically construct \itemph{some} non-trivial partition
$\mathfrak C$ in the hope that repeated further applications of \func{Balance}, 
\func{Simplify}, and \func{Reduce} to its members in the main loop of
Algorithm~\ref{alg:main} will eventually produce regular toric data only.
Success of this procedure is not guaranteed and we need to
allow \func{Reduce} to fail (at which point Algorithm~\ref{alg:main}
will fail too) in order to guarantee termination.

\paragraph{Reduction candidates.}
We begin by isolating a source of the singularity of $\cT$.
Namely, the method for regularity testing in \S\ref{ss:regular} readily provides
us with an inclusion-minimal set $J \subset\{1,\dotsc,r\}$ such 
that the Jacobian matrix of $(\init_{\cC_0}(f_j))_{j\in J}$ has rank
less than $\card J$ at some point $\uu\in \Torus^n(\bar k)$ with $\init_{\cC_0}(f_j)(\uu)
= 0$ for all $j\in J$.
After renumbering $f_1,\dotsc,f_r$, we may assume that $J = \{ 1,\dotsc,e\}$.
If $e = 1$, then we give up and let \func{Reduce} fail.
Suppose that $e \ge 2$.
By a \emph{reduction candidate} for $\cT$ we mean a quadruple
$(i,j,t_i,t_j)$, where $1\le i < j \le e$, 
$t_i$ is a term of $\init_{\cC_0}(f_i)$,
and $t_j$ is a term of $\init_{\cC_0}(f_j)$.

\paragraph{Performing reduction.}
Let $(i,j,t_i,t_j)$ be a reduction candidate for $\cT$.
Let $\alpha_i$ and $\alpha_j$ denote the exponent vectors of the monomial in
$t_i$ and $t_j$, respectively.
We decompose $\cC_0$ into two pieces
$\cC_0^{\le} := \cC_0 \cap \{ 
\alpha_j-\alpha_i\}^*$ and $\cC_0^> := \cC_0^{\phantom =}\!\setminus
\cC_0^{\le}$, both of which are themselves rational half-open cones.
Note that the restriction of the linear form $\bil{\alpha_i}\blank$ to $\cC_0$
only depends on $f_i$ and not on the chosen term $t_i$, and similarly for
$\alpha_j$.
In particular, the decomposition $\cC_0^{\phantom=} = \cC_0^{\le} \cup \cC_0^{>}$ 
only depends on $(i,j)$.
Define toric data
\begin{align*}
\cT^\le &
:= \Bigl(\cC_0^\le; \,\,
f_1,\dotsc,f_{j-1}, \,\, f_j - \frac{t_j}{t_i} f_i, \,\, f_{j+1}, \dotsc,f_r
\Bigr),
\\
\cT^> &
:= \Bigl(\cC_0^>; \,\,
f_1,\dotsc,f_{i-1}, \,\, f_i - \frac{t_i}{t_j} f_j, \,\,
f_{i+1},\dotsc,f_r\Bigr).
\end{align*}

By construction,  $t_j/t_i \in k\bigl[(\cC_0^\le)^*\cap\ZZ^n\bigr]$
and $t_i/t_j\in k\bigl[(\cC_0^>)^*\cap\ZZ^n\bigr]$ so Lemma~\ref{lem:eqop}(\ref{S2})
shows that $\{ \cT^\le, \cT^>\}$ is a partition of $\cT$.
Having chosen a reduction candidate $(i,j,t_i,t_j)$,
we let $\func{Reduce}(\cT)$ return $\{ \cT^\le, \cT^>\}$.

The name ``reduction'' given to the procedure described here is due to the
similarity to reduction steps in the theory of Gr\"obner bases, see e.g.\
\cite[\S 1.5]{AL94}.
There are, however, substantial differences between the two procedures.
Most importantly, the role of divisibility relations ``$\divides{t_i}{t_j}$''
between terms in polynomial algebras in the classical setting is replaced by an
integrality condition ``$t_j/t_i \in k[\cC_0^*\cap\ZZ^n]$'' for Laurent terms.
In the present setting, we can enforce arbitrary divisibility relations of this form
by cutting $\cC_0$ in half---at the cost of having to consider the
opposite relation as well.

\paragraph{Finding reduction candidates.}
It remains to explain a strategy for choosing a reduction candidate
$(i,j,t_i,t_j)$ for $\cT$.
This is the most critical part of the entire reduction step and it may well
fail.
We use a greedy approach.
Define the \emph{weight} of a balanced toric datum $\cT' =
(\cD_0;g_1,\dotsc,g_u)$ to be $\weight(\cT') := \sum_{d=1}^u
\noof{\supp\bigl(\init_{\cD_0}(g_d)\bigr)}$.
For each reduction candidate $(i,j,t_i,t_j)$, we construct the associated
partition $\{\cT^{\le},\cT^>\}$ of $\cT$ as indicated
above.
Using \func{Balance} and \func{Simplify}, we then further refine this partition
to produce a partition, $\mathfrak C(i,j,t_i,t_j)$ say, of $\cT$ which consists of
balanced and simple toric data.
Let $\mathfrak C'(i,j,t_i,t_j)\subset\mathfrak C(i,j,t_i,t_j)$ be the subset of
singular toric data.
If $\mathfrak C'(i,j,t_i,t_j) = \emptyset$ for some $(i,j,t_i,t_j)$,
then we use such a quadruple as our reduction candidate.
Otherwise, we choose $(i,j,t_i,t_j)$ such that
$\bigl(\sum_{\cT'\in \mathfrak C'(i,j,t_i,t_j)}\weight(\cT')\bigr)/\card{\mathfrak C'(i,j,t_i,t_j)}$ is minimal.
In practice, we then of course let $\func{Reduce}(\cT)$ return $\mathfrak
C(i,j,t_i,t_j)$ instead of $\{ \cT^{\le},\cT^>\}$.

In order to ensure termination of Algorithm~\ref{alg:main},
we assign a ``depth'' to each toric datum. 
The initial toric datum given as the input of Algorithm~\ref{alg:main}
has depth $0$. We further let $\func{Balance}(\cT)$ and $\func{Simplify}(\cT)$
return toric data of the same depth as $\cT$.
If, having chosen $(i,j,t_i,t_j)$ as part of the reduction step,
we have $\weight(\cT') > \weight(\cT)$ for some $\cT' \in \mathfrak
C'(i,j,t_i,t_j)$, then we increase the depth of $\cT'$. 
Termination is guaranteed by letting $\func{Reduce}(\cT)$ fail whenever the
depth of $\cT$ exceeds some constant value.
While this approach is less elegant than a strictly greedy approach, where we
would e.g.\ insist that $\weight(\cT') < \weight(\cT)$ for all $\cT'\in
\mathfrak C'(i,j,t_i,t_j)$, it is more powerful in practice.

\paragraph{}
For an illustration of the reduction step applied to a ``real-life'' example,
see \S\ref{ss:Fil4}.

\section{Practical matters}
\label{s:imp}

\subsection{Introducing ``\textsf{Zeta}''}

The \href{http://www.python.org/}{Python}-package
\href{http://www.math.uni-bielefeld.de/rossmann/Zeta/}{\textsf{Zeta}} 
\cite{zeta} for \href{http://sagemath.org/}{\Sage} \cite{Sage} 
provides an implementation of Algorithm~\ref{alg:main} 
for computing topological subalgebra, ideal, and submodule zeta functions for $k = \QQ$.
\Sage{} natively supports computations with rational polyhedra
and we use these capabilities to simulate computations with half-open cones, see \S\ref{ss:models}.
Polynomial arithmetic and Gr\"obner bases computations are
handled by \href{http://www.singular.uni-kl.de/}{\Singular} \cite{Singular}.
For the computations of mixed volumes mentioned in \S\ref{ss:euler}, we use
\href{http://home.imf.au.dk/jensen/software/gfan/gfan.html}{\Gfan} \cite{gfan}.
While \Sage{} does provide functionality for computing triangulations,
we use the fast implementation provided by
\href{http://www.home.uni-osnabrueck.de/wbruns/normaliz/}{\Normaliz}
\cite{Normaliz} if it is available.
In order to use  \textsf{Zeta} to compute topological zeta functions
associated with nilpotent groups via Theorem~\ref{thm:nilpotent},
one may use the \textsf{GAP}-package Guarana \cite{guarana,GAP4} which provides
an effective version of the Mal'cev correspondence.

\subsection{On the scope of Algorithm~\ref{alg:main} and its implementation}

\paragraph{Theoretical limitations.}
While Algorithm~\ref{alg:main} allows us to compute far more topological zeta
functions than \cite[Thm~6.7]{topzeta} alone could, it is fairly easy to produce 
examples that seem completely resistant to our approach.
For example, while the vast majority of known topological and local
subalgebra and ideal zeta functions arise from nilpotent Lie rings, 
to the author's knowledge, not a single example of any such zeta function
associated with a nilpotent Lie ring of class $\ge 5$ has ever been computed.
In particular, there are various examples of nilpotent Lie rings of additive
rank $6$ whose topological and local subring and ideal zeta functions 
remain unknown---our method has so far been unable to remedy this.
As the additive rank of the non-associative ring under consideration
increases or the assumption that it be nilpotent and Lie is relaxed,
examples amenable to our method become rare.

\paragraph{Practical issues.}
So far, the most successful applications of \textsf{Zeta} were concerned with
(nilpotent) associative, commutative, or Lie rings of additive rank at most $6$.
Even in the case of nilpotent Lie rings of rank $6$,
some of the computations carried out by the author took several
months to complete (using $16$ parallel processes on an ordinary computer, see \S\ref{ss:impET}).
In such cases, the most expensive step in Algorithm~\ref{alg:main} is the final line.
At this point, \id{regular} will be populated with possibly thousands of regular
toric data.
For each $(\cD_0;g_1,\dotsc,g_e) \in \id{regular}$, we then consider each of the 
half-open cones $\cD_0^J$ indexed by $J\subset\{1,\dotsc,e\}$ yielding perhaps
tens of thousands of half-open cones in total.
Finally, the triangulation step in \func{EvaluateTopologically} will often
decompose each $\bar\cD_0^J$ into possibly tens or even hundreds of thousands of
simplicial cones; note that for examples of rank $6$, the ambient Euclidean space of each
$\cD_0^J$ will have dimension at least $21 = \frac{ 6\dtimes 7}2$.

\subsection{\textrm{\func{EvaluateTopologically}} in practice}
\label{ss:impET}

As mentioned in Remark~\ref{r:ET}(\ref{r:ET1}),
in our implementation of \func{EvaluateTopologically} we immediately apply
specialisations of the form $\ess_j \mapsto \ess-j$ (see
Remark~\ref{r:special}) needed to recover the desired univariate topological zeta
function. Moreover, in order to avoid costly rational function arithmetic, 
we do not actually carry out either the summation in line~\ref{alg:summing_up}
of Algorithm~\ref{alg:summing_up} nor that in the proof of
Lemma~\ref{lem:compute_reduction}.
Instead, we first compute the final output of Algorithm~\ref{alg:main},
$\Zeta_{\topo}(\ess)\in \QQ(\ess)$ say, as an
unevaluated (possibly large) sum of rational functions of the form
$\frac{c}{(a_1 \ess - b_1) \dotsb (a_d \ess - b_d)}$ for suitable integers
$a_i,b_i,c$ (arising from simplicial cones in Lemma~\ref{lem:compute_reduction}
and Euler characteristics in \S\ref{ss:euler}).
As we construct these rational functions, we keep track of a ``candidate
denominator'' of $\Zeta_{\topo}(\ess)$, i.e.\!\ a polynomial $g\in \ZZ[\ess]$,
$g\not= 0$ with $g\Zeta_{\topo}(\ess)\in \ZZ[\ess]$. 
Using Lemma~\ref{lem:degbound}, we may then recover $\Zeta_{\topo}(\ess)$ using
random evaluation and polynomial interpolation.

After successful termination of the main loop (lines~\ref{alg:init_main}--\ref{alg:end_main_loop}) in
Algorithm~\ref{alg:main}, the remaining tasks of computing Euler
characteristics, triangulating cones, and evaluating rational functions can be
trivially parallelised and our implementation makes use of this.

\subsection{Computing with half-open cones}
\label{ss:models}

We defined our basic data structure, the toric data from
\S\ref{s:tdata}, in terms of rational half-open cones $\cC_0\subset \Orth^n$ since they 
are the smallest collection of subsets of $\RR^n$ which contains $\Orth^n$ and
which is stable under the effects of \func{Balance}, \func{Simplify}, and \func{Reduce}.
However, half-open cones and polyhedra (rational or not) are scarcely used in the literature and
they are usually not directly supported by existing software.
Apart from triangulating closed rational cones,
the only computational tasks involving half-open cones that we actually relied upon
are the following:
\begin{enumerate}
\item
  Compute the intersection of two rational half-open cones.
\item
  Decide if a rational half-open cone is empty.
\item\label{ophoc3}
  Construct the closure of a non-empty rational half-open cone.
\item\label{ophoc4}
  Decide membership of a vector in the dual of a rational half-open cone.
\end{enumerate}
As $\cC_0^* = \bar\cC_0^*$ for a half-open cone $\cC_0$,
(\ref{ophoc3}) reduces~(\ref{ophoc4}) to the closed case (which is standard).

\begin{defn}
  By a (polyhedral) \emph{model} of a rational half-open cone $\cC_0\subset\RR^n$,
  we mean a rational polyhedron $\cP_0\subset\RR^n$ such that $\cC_0\cap \ZZ^n =
  \cP_0 \cap \ZZ^n$
  and $\NN \cP_0 = \cP_0$. 
\end{defn}

For example, for each $a\in \NN$, the closed
interval $[1/a,\infty)$ is a model of the open interval $(0,\infty)$.
As we will now explain, we may replace half-open cones by models in
our computations.

\begin{prop}
  \label{p:cone_models}
  Let $\cC_0^{\phantom\prime}, \cC_0'\subset \RR^n$ be rational half-open  cones.
  \begin{enumerate}
  \item
    \label{p:cone_models1}
    $\cC_0$ admits a model.
  \item
    \label{p:cone_models2}
    Let $\cP_0^{\phantom\prime}\subset\cC_0^{\phantom\prime}$ and
    $\cP_0'\subset\cC_0'$ be models.
    Then $\cP_0^{\phantom\prime}\cap\cP_0'$ is a model of
    $\cC_0^{\phantom\prime}\cap \cC_0'$.
  \item
    \label{p:cone_models4}
    Let $\cP_0$ be a model of $\cC_0$.
    Then $\cC_0$ is empty if and only if $\cP_0$ is empty.
  \item
    \label{p:cone_models3}
    If $\cC_0\not= \emptyset$ and $\cP_0$ is a model of $\cC_0$, then
    $\bar\cC_0$ is the smallest cone containing $\cP_0$.
  \item
    \label{p:cone_models5}
    If $\cC_0^{\phantom\prime} \cap \ZZ^n = \cC_0'\cap\ZZ^n$,
    then $\cC_0^{\phantom\prime} = \cC_0^\prime$. 
    Hence, $\cC_0$ is determined by any of its models.
  \end{enumerate}
\end{prop}
\begin{proof}
  \quad
  \begin{enumerate}
    \item There are finitely many $\phi_i,\chi_j \in \ZZ^n$ ($i\in
      I,j\in J$) with $\cC_0 = \{ \omega \in \RR^n : \forall i\in I. \bil{\phi_i}\omega \ge 0,
      \forall j\in J. \bil{\chi_j}\omega > 0\}$.
      Hence,
      $\{ \omega \in \RR^n : \forall i\in I. \bil{\phi_i}\omega \ge 0,
      \forall j\in J. \bil{\chi_j}\omega \ge 1\}$ is a model of $\cC_0$.
    \item
      Obvious.
    \item
      The relative interior $\relint(\cC)$ of a non-empty
      rational cone $\cC\subset\RR^n$ satisfies
      $\relint(\cC)\cap\QQ^n\not=\emptyset$ (in fact, $\relint(\cC)\cap\QQ^n$ is
      dense in $\cC$) and hence $\relint(\cC)\cap\ZZ^n\not=\emptyset$.
      Hence, if $\cP_0 = \emptyset$, then $\cC_0=\emptyset$.
      Suppose that $\cP_0\not=\emptyset$.
      There exists $\beta \in \cP_0 \cap\QQ^n$
      and thus $a\beta\in \cP_0\cap\ZZ^n = \cC_0\cap\ZZ^n$ for some $a\in \NN$ whence
      $\cC_0\not=\emptyset$.
    \item
      For $\omega \in \cC_0\cap\QQ^n$, there exists $a\in \NN$ with $a\omega\in
      \cC_0\cap\ZZ^n = \cP_0 \cap\ZZ^n\subset\cP_0$.
      Hence, every cone containing $\cP_0$ also contains 
      $\cC_0 \cap\QQ^n$, a dense subset of $\bar\cC_0$.
      Similarly, $\cP_0 \cap \QQ^n \subset \bar\cC_0$ so that $\cP_0 \subset \bar\cC_0$.
    \item
      By (\ref{p:cone_models4}),
      we may assume that $\cC_0^{\phantom\prime}\cap\ZZ^n = \cC_0' \cap\ZZ^n\not= \emptyset$.
      Let $\cP_0$ be a model of $\cC_0^{\phantom\prime}$.
      Then $\cP_0$ is also a model of both
      $\cC_0'$ and $\cC_0'' :=\cC_0^{\phantom\prime}\cap \cC_0'$.
      By (\ref{p:cone_models3}),
      $\cC_0^{\phantom\prime}, \cC_0', \cC_0''$
      all have the same closure,  $\cC$ say.
      Suppose that $\cC_0''\not= \cC_0^{\phantom\prime}$.
      Then there exists a face $\tau$ of $\cC$ such that $\relint(\tau) \subset
      \cC_0^{\phantom\prime}$ but $\tau \cap \cC_0''= \emptyset$.
      Let $\omega \in\relint(\tau)\cap \ZZ^n$. Then
      $\omega \in \cC_0^{\phantom\prime}\cap \ZZ^n = \cC_0''\cap \ZZ^n$, a contradiction.
      Hence, $\cC_0^{\phantom\prime} \subset \cC_0'$ and so
      $\cC_0^{\phantom\prime} = \cC_0'$ by symmetry.
      \qedhere
  \end{enumerate}
\end{proof}

Given a model $\cP_0$ of $\cC_0$, it remains to recover $\bar\cC_0$ explicitly
from $\cP_0$.

\begin{lemma}
  \label{lem:mincone}
  Let $\cQ = \conv(\alpha_1,\dotsc,\alpha_e) \subset\RR^n$ be a non-empty
  polytope and $\cD\subset\RR^n$ be a cone.
  Write $\cC = \cone(\alpha_1,\dotsc,\alpha_e)$.
  Then $\cC + \cD$ is the smallest cone containing $\cQ+\cD$.
\end{lemma}
\begin{proof}
  Clearly,
  $\cC + \cD$ is a cone containing $\cP := \cQ + \cD$.
  Let $\cB\supset \cP$ be a cone.
  Then $\cQ \subset \cB$ and hence $\cC \subset \cB$.
  Fix $\xx \in \cQ$ and let $\yy\in \cD$.
  Then $\xx + a\yy\in \cP$ for $a > 0$ whence $a^{-1} \xx + \yy\in \cB$.
  As $\cB$ is closed, we conclude that $\yy\in \cB$ and thus $\cD \subset \cB$.
\end{proof}

Recall that (rational) polyhedra in $\RR^n$ are exactly the sets of
the form $\cQ + \cD$, where $\cQ\subset\RR^n$ is a (rational) polytope and
$\cD\subset\RR^n$ is a (rational) cone, see e.g.\ \cite[Thm~4.13]{Bar08}.
Let $\cP_0$ be a model of a non-empty rational half-open cone $\cC_0$.
Writing $\cP_0 = \cQ + \cD$ for (rational) $\cQ,\cD$ as in Lemma~\ref{lem:mincone}, 
Proposition~\ref{p:cone_models}(\ref{p:cone_models3}) allows us to recover
$\bar\cC_0$. 
  
\section{Examples}
\label{s:app}

We discuss examples of previously unknown topological zeta functions
computed using \textsf{Zeta} \cite{zeta}.
For more examples, we refer to the database of topological subring, ideal,
and submodule zeta functions included with \textsf{Zeta}.

\subsection{Five-dimensional nilpotent Lie algebras: \texorpdfstring{$\Fil_4$}{Fil4}}
\label{ss:Fil4}

As in \cite[Thm~3.6]{Woo05}, let
$\Fil_4$ be the nilpotent Lie ring with $\ZZ$-basis $(\ee_1,\dotsc,\ee_5)$
and Lie bracket $[\ee_1, \ee_2] = \ee_3$, $[\ee_1,\ee_3] = \ee_4$, $[\ee_1,\ee_4] = \ee_5$,
$[\ee_2,\ee_3] = \ee_5$, and $[\ee_i,\ee_j] = 0$ for $i\le j$ not listed above.
As we explained in \cite[\S 7.3]{topzeta}, with the sole exception of $\Fil_4
\otimes_{\ZZ} \CC$, each of the $16$ isomorphism classes of 
non-trivial nilpotent Lie $\CC$-algebras of dimension at most $5$ admits a $\ZZ$-form whose
local subring zeta functions have been computed.
We can use Algorithm~\ref{alg:main} and \textsf{Zeta} to confirm that for the
$15$ known types, the topological zeta function coincides with the one naively
deduced from $p$-adic formulae.
The local subring zeta functions of $\Fil_4$ have so far resisted 
attempts at computing them \cite[p.\ \!57]{Woo05}.
In \cite[Eqn~(7.8)]{topzeta}, we announced that
\begin{align}
  \nonumber
  \zeta_{\Fil_4,\topo}(\ess)  =\,\, &
  \bigl(392031360 \ess^9 - 5741480808 \ess^8 + 37286908278 \ess^7 -
  140917681751 \ess^6 +
  \\&\,\, 
  \nonumber
  341501393670 \ess^5 - 550262853249 \ess^4 + 589429290044 \ess^3  -
  \\&\,\, 
  \nonumber
  404678115300 \ess^2 +
  161557332768 \ess - 
  \\&\,\,
  \nonumber
  28569052512\bigr) {\mathlarger{/}}
  \bigl(3(15\ess - 26)(7\ess - 12)(7\ess - 13)(6\ess - 11)^3
  \\&\quad\quad
  (5\ess- 8)(5\ess - 9)(4\ess - 7)^2(3\ess - 4)(2\ess - 3)(\ess - 1)\ess\bigr).
  \label{eq:Fil4}
\end{align}
For a group-theoretic interpretation,
since the topological subgroup zeta function of a torsion-free, finitely
generated nilpotent group $G$ only depends on the $\CC$-isomorphism type of
$\fL(G)\otimes_{\QQ}\CC$ (see \cite[Prop.\ 5.19(ii)]{topzeta}),
we thus obtain a complete classification of topological subgroup zeta functions
of nilpotent groups of Hirsch length at most $5$, see the database included with \textsf{Zeta}.

We will now provide details on the computation leading to \eqref{eq:Fil4}.
In doing so, we illustrate the key steps of Algorithm~\ref{alg:main}.

\paragraph{Constructing an initial toric datum.}
The first step is to construct Laurent polynomials as in Theorem~\ref{thm:coneint}(\ref{thm:coneint1})
(see the end of \S\ref{s:background_gam}) and an associated toric datum
as in Remark~\ref{r:special};
the input of Algorithm~\ref{alg:main} then consists of said toric datum,
$\cT^0$ say, and a $(0,1)$-matrix $\beta \in \Mat_{5\times 15}(\NN_0)$
(which can be easily constructed as in Remark~\ref{r:special}).

Using the defining basis $(\ee_1,\dotsc,\ee_5)$ of $\Fil_4$ and
after performing simplification steps (see \S\ref{ss:simplify}), 
we thus obtain the initial toric datum $\cT^0 = (\cC_0; \,f_1,f_2,f_3)$,
where
\begin{alignat*}{4}
\cC_0 = \Bigl\{
(\omega_1,\dotsc,\omega_{15})\in \Orth^{15} :
\omega_{10} & \le \omega_1 + \omega_6,  &\quad&
\omega_{13} & &\le \omega_1 + \omega_{10},
\\
\omega_{15} & \le \omega_6 + \omega_{10}, &\quad& 
\omega_{15} & &\le \omega_1 + \omega_{13}
\Bigr\}
\end{alignat*}
and $f_1,f_2,f_3\in \QQ[X_1^{\pm 1},\dotsc,X_{15}^{\pm 1}]$ are given by
{
\small
\begin{align*}
f_1   \,=\,\, & \underline{X^{\phantom 1}_{2} X^{\phantom 1}_{10} X_{15}^{-1} - X^{\phantom 1}_{1} X^{\phantom 1}_{10} X_{13}^{-1} X_{14} X_{15}^{-1}}
+ X^{\phantom 1}_{1} X^{\phantom 1}_{11} X_{15}^{-1} 
\\
f_2  \,=\,\,  &
\underline{X^{\phantom 1}_{1} X^{\phantom 1}_{7} X_{13}^{-1}} - X^{\phantom 1}_{1} X^{\phantom 1}_{6} X_{10}^{-1} X^{\phantom 1}_{11} X_{13}^{-1} \\
f_3 \,=\,\, &
\underline{X^{\phantom 1}_{1} X^{\phantom 1}_{7} X_{13}^{-1} X^{\phantom 1}_{14} X_{15}^{-1} - X^{\phantom 1}_{2} X^{\phantom 1}_{7} X_{15}^{-1}}
+ X^{\phantom 1}_{3} X^{\phantom 1}_{6} X_{15}^{-1}  - X^{\phantom 1}_{1} X^{\phantom 1}_{8} X_{15}^{-1} +
\\
& X^{\phantom 1}_{1} X^{\phantom 1}_{6} X_{10}^{-1} X^{\phantom 1}_{12} X_{15}^{-1} -
X^{\phantom 1}_{1} X^{\phantom 1}_{6} X_{10}^{-1} X^{\phantom 1}_{11} X_{13}^{-1} X^{\phantom 1}_{14} X_{15}^{-1};
\end{align*}
}
please ignore the underlines at first reading.

\paragraph{Balancing and regularity testing.}
Our next task is to decompose $\cC_0$ using the normal cones of $\cN :=
\Newton(f_1f_2f_3) \subset \RR^{15}$ to obtain a system of \itemph{balanced} toric data.
Using \Sage{}, we find that $\cN$ is a $6$-dimensional polytope with
$27$-vertices and a total of 395 faces; note that \Sage{} regards $\emptyset$ as a
face of $\cN$ but that we do not.
For each face $\tau\subset \cN$, we then simplify the balanced toric datum
$\Bigl(\cC_0 \cap\NormalCone_{\tau}(\cN); \,f_1,f_2,f_3\Bigr)$ (see
\S\ref{ss:simplify}) and test it for regularity (see \S\ref{ss:regular}).
It turns out that all but 4 of these 395 toric data are already regular.

\paragraph{Singularity and reduction.}
We now consider one of the aforementioned four singular toric data
arising from $(\cC_0;f_1,f_2,f_3)$ in detail, namely $(\tilde\cC_0;f_1,f_3)$,
where
\begin{alignat*}{4}
\tilde\cC_0 = \Bigl\{
(\omega_1,\dotsc,\omega_{15})\in \Orth^{15} :
\omega_1 & \le \omega_2 + \omega_{13}, &\quad \omega_{10} & \le \omega_1 + \omega_6,
\\
\omega_{13} & \le \omega_1 + \omega_{7}, &\quad \omega_{13} & \le \omega_1 + \omega_{10},
\\
\omega_{15} & \le \omega_1 + \omega_{13}, &\quad \omega_{15} & \le \omega_6 + \omega_{10},
\\
\omega_1 + \omega_{14} & = \omega_2 + \omega_{13},
\\
\omega_2 + \omega_7 & < \omega_1 + \omega_8, &\quad \omega_2 + \omega_7 & < \omega_3 + \omega_6,
\\
\omega_2 + \omega_{10} & < \omega_1 + \omega_{11}, &\quad \omega_7 + \omega_{10} & < \omega_6 + \omega_{11},
\\
\omega_2 + \omega_7 + \omega_{10} & < \omega_1 + \omega_6 + \omega_{12}
\Bigr\},
\end{alignat*}
which is a $14$-dimensional half-open cone contained in $\cC_0$.
We see that the initial forms of $f_1$, $f_2$, and $f_3$ on $\tilde\cC_0$ are exactly the
underlined parts from above.
In particular, since $\init_{\tilde\cC_0}(f_2)$ is a term,
we can see why $f_2$ has been discarded by the simplification step.

One checks that while $\init_{\tilde\cC_0}(f_1)=0$ and $\init_{\tilde\cC_0}(f_3)=0$ 
both define smooth hypersurfaces within $\Torus^{15}_{\QQ}$, 
the rank condition defining regularity is violated precisely on the subvariety
(subtorus, in fact) defined by $X_1X_{14} = X_2 X_{13}$.
Indeed, looking at the initial forms of $f_1$ and $f_3$, we see that the failure
of regularity is due to these initial forms being \itemph{identical}
up to multiplication by a unit in $\QQ[X_1^{\pm 1},\dotsc,X_{15}^{\pm 1}]$.
The reduction procedure explained in \S\ref{ss:reduce}
arose from the observation
that such geometrically simple causes of singularity are remarkably common in practice.

Reduction can  ``repair'' the failure of regularity of
$(\tilde\cC_0;f_1,f_3)$ as follows.
Define terms $t_1 := X^{\phantom 1}_{2} X^{\phantom 1}_{10} X_{15}^{-1}$ and
$t_3 := - X^{\phantom 1}_{2} X^{\phantom 1}_{7} X_{15}^{-1}$ of $f_1$ and $f_3$,
respectively; note that $t_1^{-1}t_3^{\phantom 1} = - X_7^{\phantom 1}
X_{10}^{-1}$.
Define
$\tilde\cC_0^{\le} = \{ \omega\in \tilde\cC_0 : \omega_{10}\le \omega_7\}$,
$\tilde\cC_0^{>} = \{ \omega\in\tilde\cC_0 : \omega_7 < \omega_{10}\}$,
$g_3 = f_3 - t_1^{-1}t_3^{\phantom 1} f_1$, and
$g_1 = f_1 - t_1^{\phantom 1}t_3^{-1} f_3$. Then
\begin{align*}
  g_3 \,=\,\, &
  X^{\phantom 1}_3 X^{\phantom 1}_6 X_{15}^{-1} - X^{\phantom 1}_1 X^{\phantom
    1}_8 X_{15}^{-1} + X^{\phantom 1}_1 X^{\phantom 1}_7 X_{10}^{-1} X^{\phantom
    1}_{11} X_{15}^{-1} +
  \\ & X^{\phantom 1}_1 X^{\phantom 1}_6 X_{10}^{-1} X^{\phantom 1}_{12}
  X_{15}^{-1} - X^{\phantom 1}_1 X^{\phantom 1}_6 X_{10}^{-1} X^{\phantom
    1}_{11} X_{13}^{-1} X^{\phantom 1}_{14} X_{15}^{-1} \text{ and }
  \\
  g_1 \,=\,\, &
  X^{\phantom 1}_3 X^{\phantom 1}_6 X_7^{-1} X^{\phantom 1}_{10} X_{15}^{-1} - X^{\phantom 1}_1 X_7^{-1} X^{\phantom 1}_8 X^{\phantom 1}_{10} X_{15}^{-1} + X^{\phantom 1}_1 X^{\phantom 1}_{11}
  X_{15}^{-1} + \\ &X^{\phantom 1}_1 X^{\phantom 1}_6 X_7^{-1} X^{\phantom 1}_{12} X_{15}^{-1} - X^{\phantom 1}_1 X^{\phantom 1}_6 X_7^{-1} X^{\phantom 1}_{11} X_{13}^{-1}
  X^{\phantom 1}_{14} X_{15}^{-1}.
\end{align*}
Setting $\cT^{\le} := \bigl(\tilde\cC_0^{\le}; f_1,g_3\bigr)$
and $\cT^{>} := \bigl(\tilde\cC_0^{>};g_1,f_3\bigr)$,
we obtain a  partition $\bigl\{\cT^{\le}, \cT^{>}\bigr\}$
of $\bigl(\tilde\cC_0;f_1,f_3\bigr)$.
As we will now explain, our particular choice of a reduction candidate
(terminology as in \S\ref{ss:reduce}) eliminates the source of the singularity
of $\bigl(\tilde\cC_0;f_1,f_3\bigr)$ that we isolated above. 
In general, it is possible for choices of reduction candidates to
introduce new singularities.
In order to verify that this is not the case here, we apply the balancing
procedure from \S\ref{ss:balance} followed by simplification to both $\cT^{\le}$
and to $\cT^{>}$.
In doing so, each of these toric data is partitioned into $31$
balanced conditions, $31$ being the number of faces of the Newton
polytope of $g_1$ and of $g_3$.
Fortunately, every single one the resulting $62$ toric data is now
regular, which concludes our efforts regarding $\bigl(\tilde\cC_0;f_1,f_3\bigr)$. 

The other three singular toric data mentioned above can be handled in
a very similar way.
Each of them is cut in two by reduction and each piece is then decomposed into
$15$ regular conditions by \func{Balance}.
In particular, in each case, a single application of the reduction step followed
by balancing and simplification immediately yields regular conditions
only---this is not to be expected in general.

After completion of the main loop in Algorithm~\ref{alg:main}, we have constructed
a total of $543$ regular toric data constituting a partition of $\cT^0$ from above.

\paragraph{Final stage.}
It remains to apply the function \func{EvaluateTopologically} to each of the $543$
aforementioned toric data and to recover $\zeta_{\Fil_4,\topo}(\ess)$
from a sum of rational functions.
As it is unlikely  to offer any new insights, we chose not to give details
on the tedious acts of computing Euler characteristics, triangulating cones,
and manipulating rational functions that constitute this step.

\paragraph{Stats.}
We briefly indicate the extent to which practical applications of
Algorithm~\ref{alg:main} rely on machine computations.
Thus, using \textsf{Zeta} on an Intel Xeon E5-2670 ($8$~cores)
running Sage~6.3, the computation of $\zeta_{\Fil_4,\topo}(\ess)$ sketched above
took about $97$ minutes in total.
The main loop in Algorithm~\ref{alg:main} was completed after $2$~minutes;
the vast majority of time was then spent in the final line of
Algorithm~\ref{alg:main} which used $16$ parallel processes, see \S\ref{ss:impET}.
Using polynomial interpolation, the final formula \eqref{eq:Fil4} was then
recovered from a sum of {12,869,940} rational functions as explained in
\S\ref{ss:impET}.

\subsection{Other examples}

There are various interesting examples which are similar to $\Fil_4$ in the
sense that a single application of the reduction step to singular toric data
already suffices.
One such example is given by the topological submodule zeta function of the full
unipotent group $\mathrm U_5(\ZZ) \le \GL_5(\ZZ)$ acting on its natural module 
discussed in \cite[\S 7.3]{topzeta}.
For a non-nilpotent, commutative, and associative example,
we find the topological subring zeta function of $\ZZ[X]/X^4$ to be
{\small
\[
\zeta_{\ZZ[X]/X^4,\topo}(\ess) = 
\frac{2021760 s^5 - 8509620 s^4 + 14322332 s^3 - 12036071 s^2 + 5044460 s - 842400}
{168480 (6s - 5) (4s - 3) (s - 1)^6 s}.
\]
}
A computation of similar overall complexity which however requires multiple
iterations of reduction yields the formula for
$\zeta_{\mathfrak{gl}_2(\ZZ),\topo}(\ess)$ announced in \cite[(7.7)]{topzeta}.
For a more complicated example,
consider $\Fil_4 \oplus (\ZZ,0)$, where $(\ZZ,0)$ denotes $\ZZ$ regarded as an
abelian Lie ring.
After about 4~days (same machine as for $\Fil_4$ above), \textsf{Zeta}
reports that
{
\small
\begin{align}
\nonumber
\zeta_{\Fil_4 \oplus(\ZZ,0),\topo}(\ess) =\,\, &
\bigl(52839554826240 s^{15} - 1612571385729024 s^{14} +
\nonumber
\\&\,\, 22945067840268288 s^{13} - 201917310138409536 s^{12} +
\nonumber
\\&\,\, 1228942670032455984 s^{11} - 5479610770178424720 s^{10} +
\nonumber
\\&\,\, 18489925054934205732 s^9 - 48077179247205683304 s^8 +
\nonumber
\\&\,\, 97118269735864324559 s^7 - 152405042677332499112 s^6 +
\nonumber
\\&\,\, 184268407184801648476 s^5 - 168562287295854189878 s^4 +
\nonumber
\\&\,\, 112921211241642321545 s^3 - 52295417007047312650 s^2 +
\nonumber
\\&\,\, 14969814525806597400 s - 1996549752637440000\bigr)
{\mathlarger{/}}
\nonumber
\\&\quad\quad
(48 (15 s - 31) (15 s - 34) (13 s - 28) (12 s - 25) (9 s - 20) (7 s - 15) 
\nonumber
\\&\quad\quad
(7 s - 16)(6 s - 13)^3 (5 s - 11) (4 s - 9)^3 (3 s - 5) (s - 1) (s- 2)^4 s), 
\label{eq:Fil4ZZ}
\end{align}
}

For further examples of topological zeta functions of the type considered in
this article, see \textsf{Zeta} and the database that comes with it.

\subsection{On the reliability of our computations}

When it comes to trusting computer output such as the examples given above, 
caution is certainly warranted.
Apart from possible bugs in the author's code,
the sheer number of mathematical libraries and programs relied upon by \textsf{Zeta} is a natural source of concern.

\paragraph{Independent confirmation.}
As a simple test, we can use the many examples of local zeta functions
computed by Woodward and others and compare the associated ``naive''
topological zeta functions (obtained via symbolic expansion in $p-1$ as
indicated in the introduction) with the ones obtain using \textsf{Zeta}, assuming our method applies.
Our implementation passes this test for all examples from \cite{dSW08} that we
considered.
Conversely, our machine computations thus provide evidence for the
correctness of these formulae which were often obtained using complicated,
at least partially manual, and often undocumented computations.

\paragraph{Conjectures.}
For genuinely new examples such as the topological subring zeta functions of
$\Fil_4$ and $\Fil_4 \oplus (\ZZ,0)$, we regard the peculiar
conjectural features of topological zeta functions from \cite[\S 8]{topzeta}
as further evidence of the reliability of our implementation---indeed,
computational errors can easily destroy these properties. 
The lengthy formula \eqref{eq:Fil4ZZ}, for example, has all the properties
predicted by the conjectures in \cite[\S 8]{topzeta}.
In addition, we observe that $\zeta_{\Fil_4,\topo}(\ess)$ and $\zeta_{\Fil_4
  \oplus (\ZZ,0),\topo}(\ess)$ ``agree at infinity'' in the following sense.
Given any non-associative ring $\cA$ of additive rank $d$,
let $\magic(\cA) := \zeta_{\cA,\topo}({\ess}^{-1}) {\ess}^{d} \big\vert_{\ess=0}$.
The ``degree conjecture'' \cite[Conj.\ I]{topzeta} asserts that $0\not=
\magic(\cA) \not= \infty$. 
\addtocounter{conj}{4}
\begin{conj}
  \label{conj:magic}
  $\magic(\cA) = \magic(\cA \oplus (\ZZ,0))$.
\end{conj}
For instance, \eqref{eq:Fil4} and \eqref{eq:Fil4ZZ} show that
$\magic(\Fil_4) = \magic(\Fil_4 \oplus (\ZZ,0)) = 463/1350$. 

The effect of the operation $\cA\mapsto \cA \oplus(\ZZ,0)$ (let alone arbitrary
direct sums) on subring or ideal zeta functions (local or topological) is
poorly understood in general. For a few specific examples of rings $\cA$,
formulae for local zeta functions of $\cA 
\oplus(\ZZ^r,0)$ are known for all $r\ge 0$, see \cite{dSW08};
these formulae are consistent with Conjecture~\ref{conj:magic}.
The simple patterns exhibited by the formulae for known instances of such
families seem to be exceptional, as e.g.\ suggested by various examples of
topological zeta functions included with \textsf{Zeta}.
The experimental evidence underpinning Conjecture~\ref{conj:magic} is all 
the more remarkable in view of the generally increased complexity of
$\zeta_{\cA\oplus(\ZZ,0),\topo}(\ess)$ compared with $\zeta_{\cA,\topo}(\ess)$.

{
  \bibliographystyle{abbrv}
  \footnotesize
  \bibliography{topzeta}

\def\cprime{$'$}
\begin{bibdiv}
\begin{biblist}

\bib{AL94}{book}{
      author={Adams, William~W.},
      author={Loustaunau, Philippe},
       title={An introduction to {G}r\"obner bases},
      series={Graduate Studies in Mathematics},
   publisher={American Mathematical Society},
     address={Providence, RI},
        date={1994},
      volume={3},
        ISBN={0-8218-3804-0},
}

\bib{Alu03}{article}{
      author={Aluffi, Paolo},
       title={Computing characteristic classes of projective schemes},
        date={2003},
     journal={J. Symbolic Comput.},
      volume={35},
      number={1},
       pages={3\ndash 19},
}

\bib{guarana}{manual}{
      author={Assmann, Bj{\"o}rn},
       title={{G}uarana. {A}pplications of {L}ie methods for computations with
  infinite polycyclic groups},
        date={2012},
        note={{A {{\sf GAP} 4}-package. Available from
  \url{http://www.gap-system.org/Packages/guarana.html}}.},
}

\bib{Bar08}{book}{
      author={Barvinok, Alexander},
       title={Integer points in polyhedra},
      series={Zurich Lectures in Advanced Mathematics},
   publisher={European Mathematical Society (EMS), Z{\"u}rich},
        date={2008},
         url={http://dx.doi.org/10.4171/052},
}

\bib{Normaliz}{misc}{
      author={Bruns, W.},
      author={Ichim, B.},
      author={R\"omer, T.},
      author={S\"oger, C.},
       title={Normaliz 2.11.2. {A}lgorithms for rational cones and affine
  monoids.},
        date={2014},
        note={Available from \url{http://www.math.uos.de/normaliz/}},
}

\bib{CLS11}{book}{
      author={Cox, David~A.},
      author={Little, John~B.},
      author={Schenck, Henry~K.},
       title={Toric varieties},
      series={Graduate Studies in Mathematics},
   publisher={American Mathematical Society},
     address={Providence, RI},
        date={2011},
      volume={124},
}

\bib{Den87}{article}{
      author={Denef, Jan},
       title={On the degree of {I}gusa's local zeta function},
        date={1987},
     journal={Amer. J. Math.},
      volume={109},
      number={6},
       pages={991\ndash 1008},
         url={http://dx.doi.org/10.2307/2374583},
}

\bib{DH01}{article}{
      author={Denef, Jan},
      author={Hoornaert, Kathleen},
       title={Newton polyhedra and {I}gusa's local zeta function},
        date={2001},
     journal={J. Number Theory},
      volume={89},
      number={1},
       pages={31\ndash 64},
         url={http://dx.doi.org/10.1006/jnth.2000.2606},
}

\bib{DL92}{article}{
      author={Denef, Jan},
      author={Loeser, Fran{\c{c}}ois},
       title={Caract\'eristiques d'{E}uler-{P}oincar\'e, fonctions z\^eta
  locales et modifications analytiques},
        date={1992},
     journal={J. Amer. Math. Soc.},
      volume={5},
      number={4},
       pages={705\ndash 720},
         url={http://dx.doi.org/10.2307/2152708},
}

\bib{dS00}{article}{
      author={du~Sautoy, Marcus P.~F.},
       title={The zeta function of {$\mathfrak{sl}_2(\mathbb{Z})$}},
        date={2000},
     journal={Forum Math.},
      volume={12},
      number={2},
       pages={197\ndash 221},
         url={http://dx.doi.org/10.1515/form.2000.004},
}

\bib{dSG00}{article}{
      author={du~Sautoy, Marcus P.~F.},
      author={Grunewald, Fritz~J.},
       title={Analytic properties of zeta functions and subgroup growth},
        date={2000},
     journal={Ann. of Math. (2)},
      volume={152},
      number={3},
       pages={793\ndash 833},
         url={http://dx.doi.org/10.2307/2661355},
}

\bib{dSL04}{article}{
      author={du~Sautoy, Marcus P.~F.},
      author={Loeser, Fran{\c{c}}ois},
       title={Motivic zeta functions of infinite-dimensional {L}ie algebras},
        date={2004},
     journal={Selecta Math. (N.S.)},
      volume={10},
      number={2},
       pages={253\ndash 303},
         url={http://dx.doi.org/10.1007/s00029-004-0361-y},
}

\bib{dSW08}{book}{
      author={du~Sautoy, Marcus P.~F.},
      author={Woodward, Luke},
       title={Zeta functions of groups and rings},
      series={Lecture Notes in Mathematics},
   publisher={Springer-Verlag, Berlin},
        date={2008},
      volume={1925},
         url={http://dx.doi.org/10.1007/978-3-540-74776-5},
}

\bib{Singular}{techreport}{
      author={Greuel, G.-M.},
      author={Pfister, G.},
      author={Sch\"onemann, H.},
       title={{\sc Singular} 3.0},
        type={{A Computer Algebra System for Polynomial Computations}},
 institution={Centre for Computer Algebra},
     address={University of Kaiserslautern},
        date={2005},
        note={Available from \url{http://www.singular.uni-kl.de/}},
}

\bib{GSS88}{article}{
      author={Grunewald, Fritz~J.},
      author={Segal, Dan},
      author={Smith, Geoff~C.},
       title={Subgroups of finite index in nilpotent groups},
        date={1988},
     journal={Invent. Math.},
      volume={93},
      number={1},
       pages={185\ndash 223},
}

\bib{Hel14}{article}{
      author={Helmer, Martin},
       title={An algorithm to compute the topological {E}uler characteristic,
  {C}hern-{S}chwartz-{M}ac{P}herson class and {S}egre class of projective
  varieties (preprint)},
        date={2014},
        note={\href{http://arxiv.org/abs/1402.2930}{arXiv:1402.2930}},
}

\bib{Igu00}{book}{
      author={ichi Igusa, Jun},
       title={An introduction to the theory of local zeta functions},
      series={AMS/IP Studies in Advanced Mathematics},
   publisher={Providence, RI: American Mathematical Society},
        date={2000},
      volume={14},
}

\bib{gfan}{misc}{
      author={Jensen, Anders~N.},
       title={{G}fan, a software system for {G}r{\"o}bner fans and tropical
  varieties},
        date={2011},
        note={Available from
  \url{http://home.imf.au.dk/jensen/software/gfan/gfan.html}},
}

\bib{Jos13}{article}{
      author={Jost, Christine},
       title={An algorithm for computing the topological {E}uler characteristic
  of complex projective varieties},
        date={2013},
        note={\href{http://arxiv.org/abs/1301.4128}{arXiv:1301.4128}},
}

\bib{Kho77}{article}{
      author={Khovanskii, Askold~G.},
       title={Newton polyhedra, and toroidal varieties},
        date={1977},
     journal={Funkcional. Anal. i Prilo\v zen.},
      volume={11},
      number={4},
       pages={56\ndash 64, 96},
}

\bib{Oka97}{book}{
      author={Oka, Mutsuo},
       title={Non-degenerate complete intersection singularity},
      series={Actualit\'es Math\'ematiques.},
   publisher={Hermann, Paris},
        date={1997},
        ISBN={2-7056-6343-6},
}

\bib{topzeta}{article}{
      author={Rossmann, Tobias},
       title={Computing topological zeta functions of groups, algebras, and
  modules, {I} (preprint)},
        date={2014},
        note={\href{http://arxiv.org/abs/1405.5711}{arXiv:1405.5711}},
}

\bib{zeta}{manual}{
      author={Rossmann, Tobias},
       title={{\textsf{\textup{Zeta}}, version 0.1}},
        date={2014},
        note={See \url{http://www.math.uni-bielefeld.de/~rossmann/Zeta/}.},
}

\bib{Sol77}{article}{
      author={Solomon, Louis},
       title={Zeta functions and integral representation theory},
        date={1977},
     journal={Advances in Math.},
      volume={26},
      number={3},
       pages={306\ndash 326},
}

\bib{Sta12}{book}{
      author={Stanley, Richard~P.},
       title={Enumerative combinatorics. {V}olume 1},
     edition={Second},
      series={Cambridge Studies in Advanced Mathematics},
   publisher={Cambridge University Press, Cambridge},
        date={2012},
      volume={49},
        ISBN={978-1-107-60262-5},
      review={\MR{2868112}},
}

\bib{Sage}{manual}{
      author={Stein, W.~A.},
      author={others},
       title={{S}age {M}athematics {S}oftware ({V}ersion 6.3)},
organization={The Sage Development Team},
        date={2014},
        note={Available from \url{http://www.sagemath.org/}},
}

\bib{GAP4}{manual}{
       title={{\textsf{\textup{GAP}} -- {G}roups, {A}lgorithms, and
  {P}rogramming, version 4.7.5}},
organization={The \textsf{\textup{GAP}}~Group},
        date={2014},
        note={Available from \url{http://www.gap-system.org/}},
}

\bib{VZG08}{article}{
      author={Veys, Willem},
      author={Z{\'u}{\~n}iga-Galindo, W.~A.},
       title={Zeta functions for analytic mappings, log-principalization of
  ideals, and {N}ewton polyhedra},
        date={2008},
     journal={Trans. Amer. Math. Soc.},
      volume={360},
      number={4},
       pages={2205\ndash 2227},
         url={http://dx.doi.org/10.1090/S0002-9947-07-04422-4},
}

\bib{zfarchive}{manual}{
      author={Woodward, Luke},
       title={Zeta functions of {L}ie rings archive},
        note={See \url{http://www.lack-of.org.uk/zfarchive/}.},
}

\bib{Woo05}{thesis}{
      author={Woodward, Luke},
       title={Zeta functions of groups: computer calculations and functional
  equations},
        type={Ph.D. Thesis},
        date={2005},
}

\end{biblist}
\end{bibdiv}
}

\end{document}